\documentclass[10pt]{amsart}

\usepackage[T1]{fontenc}  
\usepackage[utf8]{inputenc} 
\usepackage{longtable}
\usepackage{amssymb}
\usepackage{amsthm}
\usepackage{amsmath}
\usepackage{wasysym}
\usepackage{dcpic, pictex}
\usepackage{bbm}
\usepackage{enumerate}
\usepackage{amsfonts}
\usepackage{amsthm}
\usepackage{amsmath}
\usepackage{times}
\usepackage{amscd}
\usepackage[mathscr]{eucal}
\usepackage{indentfirst}
\usepackage{pict2e}
\usepackage{epic}
\usepackage{epstopdf} 
\usepackage{verbatim}
\usepackage{cancel}
\usepackage[foot]{amsaddr}
\usepackage{lipsum}
\usepackage{mathrsfs}
\usepackage{bm}

\usepackage{subcaption}
\usepackage{graphicx}

\usepackage{
amsfonts, 
latexsym, 
enumerate,
}
\usepackage[english]{babel}
\allowdisplaybreaks
\makeindex
\newtheorem{theorem}{Theorem}[section]
\newtheorem{lemma}[theorem]{Lemma}
\newtheorem{proposition}[theorem]{Proposition}

\newtheorem{remark}[theorem]{Remark}

\newtheorem{assumption}[theorem]{Assumption}
\usepackage[a4paper,top=3.5cm,bottom=3.5cm,left=3cm,right=3cm]{geometry}
\numberwithin{equation}{section}

\allowdisplaybreaks

\usepackage{ bbold }
\usepackage{xcolor}

\newcommand{\nc}{\normalcolor}
\newcommand{\dif}{\mathrm{d}}
\newcommand{\E}{\mathbf{E}}
\newcommand{\R}{\mathbf{R}}
\newcommand{\C}{\mathbf{C}}

\newcommand{\N}{\mathbf{N}}

\newcommand{\ii}{\mathrm{i}}
\newcommand{\ee}{\mathrm{e}}

\newcommand{\bw}{\bm{w}}
\newcommand{\bfv}{\mathbf{v}}
\newcommand{\bfw}{\mathbf{w}}

\usepackage{hyperref}

\title[Decorrelation transition in the Wigner minor process]{Decorrelation transition in the Wigner minor process}

\date{\today}

\begin{document}

\maketitle

\vspace{0.25cm}

\renewcommand{\thefootnote}{\fnsymbol{footnote}}

\noindent
\mbox{}%
\hfill% 
\begin{minipage}{0.19\textwidth}
	\centering
	{Zhigang Bao}\footnotemark[1]\\
	\footnotesize{\textit{zgbao@hku.hk}}
\end{minipage}
\hfill% 
\begin{minipage}{0.19\textwidth}
	\centering
	{Giorgio Cipolloni}\footnotemark[2]\\
	\footnotesize{\textit{gcipolloni@arizona.edu}}
\end{minipage}
\hfill%
\begin{minipage}{0.19\textwidth}
	\centering
	{L\'aszl\'o Erd\H{o}s}\footnotemark[3]\\
	\footnotesize{\textit{lerdos@ist.ac.at}}
\end{minipage}
\hfill%
\begin{minipage}{0.19\textwidth}
	\centering
	{Joscha Henheik}\footnotemark[3]\\
	\footnotesize{\textit{joscha.henheik@ist.ac.at}}
\end{minipage}
\hfill%
\begin{minipage}{0.19\textwidth}
	\centering
	{Oleksii Kolupaiev}\footnotemark[3]\\
	\footnotesize{\textit{okolupaiev@ist.ac.at}}
\end{minipage}
\hfill%
\mbox{}%
\footnotetext[1]{Department of Mathematics, The University of Hong Kong.}
\footnotetext[1]{Supported by Hong Kong RGC Grant GRF 16304724, NSFC12222121 and NSFC12271475.}
\footnotetext[2]{Department of Mathematics, University of Arizona, 617 N Santa Rita Ave, Tucson, AZ 85721, USA.}
\footnotetext[3]{Institute of Science and Technology Austria, Am Campus 1, 3400 Klosterneuburg, Austria. 
}
\footnotetext[3]{Supported by the ERC Advanced Grant ``RMTBeyond'' No.~101020331.}

\renewcommand*{\thefootnote}{\arabic{footnote}}
\vspace{0.25cm}

\begin{abstract} 
We consider the Wigner minor process, i.e. the eigenvalues of an $N\times N$ Wigner matrix $H^{(N)}$ together 
with the eigenvalues of all its $n\times n$ minors, $H^{(n)}$, $n\le N$.  The 
 top eigenvalues of  $H^{(N)}$  and those of its immediate minor $H^{(N-1)}$ are very strongly correlated, but this correlation
 becomes weaker for smaller minors $H^{(N-k)}$ as $k$ increases.
  For the GUE minor process the critical transition regime around $k\sim N^{2/3}$ was analyzed  by 
 Forrester and Nagao \cite{determinantal_corr}  
 providing an explicit formula for the nontrivial joint correlation function.   We prove that this formula is universal, i.e. 
it  holds  for the Wigner  minor process. Moreover, we give a complete analysis of the sub- and supercritical regimes
both for eigenvalues and for the corresponding eigenvector overlaps, thus we prove the decorrelation transition 
in full generality. 
\end{abstract}
\vspace{0.15cm}


\vspace{0.25cm}
\normalsize

\section{Introduction}

In the past fifteen years, research on the universality problem of random matrix spectral statistics has seen significant progress. Although most of the work focuses on the limiting properties of spectral statistics of a single random matrix, some understanding has also been gained about the relationship between spectral statistics of two dependent/coupled matrices. This includes, e.g., \cite{nagao1998multilevel} for the dynamical correlation of Dyson Brownian Motion (DBM), \cite{macroCLT_complex} for the study of weakly correlated DBMs, \cite{Bor_noise} for the eigenvector coupling/decoupling of random matrices under resampling, and \cite{nonHermdecay, eigenv_decorr} for the asymptotic independence of spectrally localized resolvents between two sufficiently distant deformations of the same random matrix. In the present work, we study a model that has a very refined understanding under the Gaussian setting but much less is known under the general setting. This is the \emph{Wigner minor process}, which we define below. 

\subsection{Model and results} 
 Let $H\equiv H^{(N)}=(h_{ij})_{N,N}$ be a $N$-dimensional Wigner matrix, i.e.~a random real symmetric or complex Hermitian matrix  with independent centered and entries of with variance $1/N$, up to the symmetry constraint $h_{ij}=\overline{h_{ji}}$. 
 For $1 \le n \le N$, we denote by $H^{(n)}$ its lower-right $n\times n$ corner and $\lambda_1^{(n)}\geq \ldots \geq \lambda_n^{(n)}$ are the ordered eigenvalues of $H^{(n)}$ with corresponding $\ell^2$-normalized eigenvectors $\bm w_1^{(n)}, ... , \bm w_n^{(n)} \in \C^n$, which we naturally identify with vectors in $\C^N$ by writing zeros in the first $N-n$ coordinates.  
Hereafter the multi-level particle system $\{(n, \lambda_{i}^{(n)}): i\in [ n], n\in [ N] \}\in [ N] \times \R$ will be called the {\it Wigner minor process}, where we used the notation $[ m]=\{1, \ldots, m\}$. There is a point at $(n, \lambda)$ if and only if the minor $H^{(n)}$ has an eigenvalue $\lambda$; cf.~Figure \ref{fig:minoreval} for illustration.

\begin{figure}[h]
\begin{center}
\includegraphics[height=5cm]{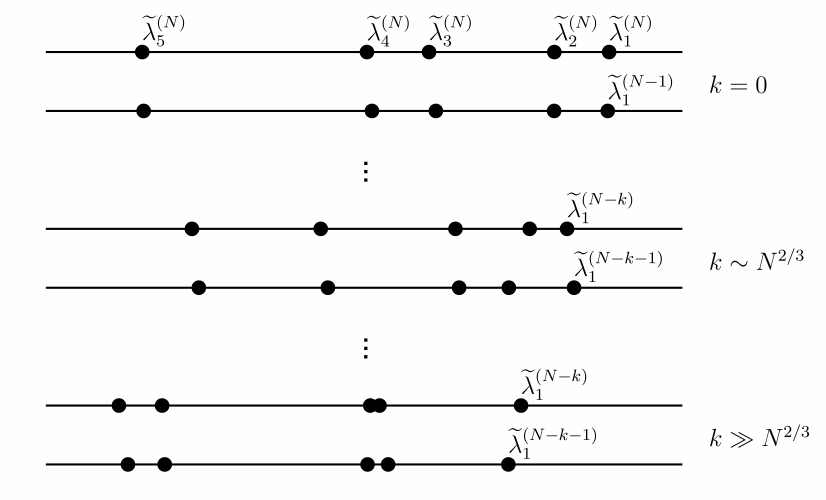}
\end{center}
\caption{Illustration of rescaled largest eigenvalues of the minors $H^{(N-k)}$. The rescaling
$\widetilde{\lambda}^{(N-k)}_j:=(1-k/N)^{-1/2}\lambda^{(N-k)}_j$ ensures that 
 $\widetilde{\lambda}^{(N-k)}_1$ remain close to 2 for all $k$.  Three regimes are depicted: the subcritical ($k=0$), the critical ($k\sim N^{2/3}$) and the supercritical ($k\gg N^{2/3}$) regime. In each regime we present two consecutive minors, highlighting the strong correlation between their top eigenvalues. In the critical regime ($k\sim N^{2/3}$), the configuration of the top eigenvalues still resembles that of the $k=0$, reflecting a non-trivial universal correlation between $\lambda_1^{(N)}$ and $\lambda_1^{(N-k)}$. In contrast, in the supercritical regime ($k\gg N^{2/3}$) the eigenvalue pattern
  deviates significantly from the subcritical case, consistent with the asymptotic independence of $\lambda_1^{(N)}$ and $\lambda_1^{(N-k)}$.}
\label{fig:minoreval}
\end{figure}

When focusing on a single level, i.e., the eigenvalues of a single $H^{(N)}$, there is a vast amount of literature 
about the universality of their local statistics. Among various results, the edge universality, also known as the \emph{Tracy-Widom law} for the extreme eigenvalues of $H^{(N)}$ has been established under rather general condition \cite{Sos_edge, Bou_extreme, Alt_band, landon2017edge},  even with an explicit rate of convergence \cite{schnelli2022convergence, schnelli2023quantitative}. 
However, beyond integrable models such as GUE, very little is known when we extend the study from a single-level system to a multi-level system, when correlations of eigenvalues from different levels become relevant.

In this paper, we establish a phase transition of the joint distribution of the top eigenvalues $\lambda_1^{(N)}$ and $\lambda_1^{(n)}$, and of the overlap of their corresponding top eigenvectors.\footnote{We focus on the top eigenvalue/eigenvector for simplicity of the presentation in the introduction; we refer to  Section \ref{sec:mainres} for the statements in full generality. Moreover, our results also hold for the bottom eigenvalues but for simplicity we focus on the top ones.}
From now on, we fix $k = k(N) := N-n$  and consider the following three regimes: 

\begin{itemize}
\item[(i)] In the \emph{supercritical regime} (see Theorem \ref{theo:large_k}), $k \gg N^{2/3}$, we prove that $\lambda_1^{(N)}$ and $\lambda_1^{(N-k)}$ are asymptotically independent, i.e.~both follow the Tracy-Widom law after suitable normalization, but independently.  Moreover, we show that the top eigenvectors are almost orthogonal, in the sense that $\big| \big\langle \bm w_1^{(N)}, \bm w_1^{(N-k)} \big\rangle \big| \ll 1$. 
\item[(ii)] In the \emph{critical regime} (see Theorem \ref{theo:crit_k}), $k \sim N^{2/3}$, we prove that the joint correlation function of $\lambda_1^{(N)}$ and $\lambda_1^{(N-k)}$ is asymptotically given by a non-trivial universal function. In the complex Hermitian case, this function has been computed by Forrester and Nagao for the \emph{Gaussian Unitary Ensemble (GUE)} \cite{determinantal_corr}. 
\item[(iii)] In the \emph{subcritical regime} (see Theorem \ref{theo:small_k}), $k \ll N^{2/3}$, we prove that the difference  $\lambda_1^{(N)}-\lambda_1^{(N-k)}$ is an asymptotically Gaussian random variable with variance of order $\sqrt{k}/N$, i.e.~much smaller than their individual fluctuation scales $N^{-2/3}$. Hence, $\lambda_1^{(N)}$ and $\lambda_1^{(N-k)}$ are strongly correlated. Moreover, we show that the top eigenvectors are almost aligned, in the sense that $ \big\langle \bm w_1^{(N)}, \bm w_1^{(N-k)} \big\rangle \approx 1$. 
\end{itemize}

We call this phase transition the \emph{decorrelation transition}, inspired by  the discovery of the 
critical level $k\sim N^{2/3}$  by Forrester and Nagao \cite{determinantal_corr} for the \emph{GUE minor process};
see Section~\ref{subsec:related} below. 
\subsection{Related works} \label{subsec:related} We now discuss several related problems studied in the literature. 

\subsubsection{The GUE minor process} 
Similarly to the single-level case, historical studies start from the complex Gaussian case, i.e., GUE, due to its accessible determinantal nature.
The GUE minor process is first studied by Baryshnikov in \cite{Bar01}, where it is shown that the largest eigenvalues of the GUE minors have the same law as the Brownian last passage percolation. It is then proved in \cite{GUE_minor_det} that the GUE minor process is determinantal, and the authors established a connection with the random tilings
of the Aztec diamond. We also refer to \cite{gorin2023six} for a connection between the six vertex model and the GUE minor process. The GUE minor process with external sources is considered in \cite{Adl_percol}. The scaling limit of the local statistics of GUE minor process and their extensions to other invariant ensembles have been studied in \cite{determinantal_corr, Bou09, metcalfe2013universality, Gor15, Gor_dbm, Naj_bead}, in both edge and bulk regimes. We also refer to \cite{adler2014dyson, Warren_dbm, Gor15, Gor_dbm, Ferrar_GUE} for a study of GUE minor process from a dynamical point of view. In addition, recent studies on the covariance estimates and universality of  eigengaps between two successive levels of GUE minor process in the bulk regime turn out to be the key technical inputs for  the analysis of random hives; see \cite{narayanan2024sums, tao2024distribution, narayanan2025limit}.  In the edge regime, the scaling limit of GUE minor process with external sources plays a fundamental role in the distribution of the eigenvectors of spiked GUE in the critical regime of BBP transition; see 
\cite{bao2022eigenvector}. \nc

Extending the result from Gaussian model to the generally distributed model and proving the universality for Wigner minor process in various local regimes, is first done in Huang's recent paper \cite{Huang_minor}. Especially, in \cite{Huang_minor},  the universality of the minor process is proved in the edge regime, when the ranks of the minors differ by a fixed number. Roughly speaking, in this case, up to the first order, the fluctuation of the  leading eigenvalues of level $n$  synchronizes with that of  level $N$ when $n=N-k$ with a fixed $k$, which is on the scale $N^{-\frac23}$ and given asymptotically by the Tracy-Widom law. The difference $\lambda_1^{(N)}-\lambda_1^{(N-k)}$ fluctuates on a much smaller scale ($\sim N^{-1}$), asymptotically given by Gamma distribution, which has been previously observed for the GUE minor process \cite{Gor_dbm}.  It is also shown in \cite{Huang_minor} that the local scaling limit for the Wigner minor process is universally given by the the bead process for general $\text{Sine}_\beta$ process in the bulk, as initially constructed for the Hermite $\beta$ ensemble in \cite{Naj_bead}. We also refer to \cite{gorin2023universal} for the edge behavior of so-called general Gaussian $\beta$ corner processes.

\subsubsection{Noise sensitivity}
Another related problem concerns the noise sensitivity of the eigenvectors of random matrices. In \cite{Bor_noise}, the authors study what will happen to the top eigenvector of a Wigner matrix if $K$ randomly chosen entries of $H^{(N)}$ are resampled  (see also \cite{bordenave2022noise} for sparse matrices)\nc. More specifically, they explore the relationship between the top eigenvector of the original Wigner matrix and the one after resampling and identify a threshold of $K\sim N^{5/3}$. In particular, the following phase transition is observed in \cite{Bor_noise}: (i) If $K\ll N^{5/3}$, the two top eigenvectors are almost colinear; (ii) If $K\gg N^{5/3}$, the two top eigenvectors are almost orthogonal. We note that the threshold here matches
 that found by Forrester and Nagao for the minor process, since $N^{2/3}$ columns/rows consist of $N^{5/3}$ entries.
  A first-order estimate on the difference between the largest eigenvalues of two matrices before and after the resampling is also provided in \cite{Bor_noise}. However, the much more delicate  relation between the two top eigenvalues  at the fluctuation level remain unexplored for this model. We expect that our approach developed for the minor process can be adapted to the resampling model as well.

\subsubsection{Law of the fractional logarithm} As the last related problem, we mention a version of the law of iterated logarithm 
for the largest eigenvalue of the Wigner matrix. More precisely, 
in \cite{PZ_LIL} and \cite{baslingker2024paquette}, a \emph{law of fractional logarithm (LFL)} for the largest eigenvalue of GUE is established, where a key ingredient is the asymptotic coupling and decoupling of the tail events of $\lambda_1^{(N)}$ and $\lambda_1^{(n)}$ in the sub- and supercritical regimes, respectively. In the decorrelation
 transition studied in this paper, we consider the relationship between the two largest eigenvalues on their typical fluctuation scale. In a separate paper \cite{LFL}, we extend these
 relations to the small/moderate deviation regime and prove LFL 
 for the largest eigenvalue of a generally distributed Wigner matrix.

\subsection{Ideas of the proofs} We now describe our proof strategy used to prove the \emph{decorrelation transition}. Each of the three regimes (subcritical, critical, and supercritical) uses quite different techniques; for this reason, we now explain the approach in each regime separately.

In the subcritical regime ($k \ll N^{2/3}$), using perturbation theory, we employ
a recursion formula which gives an expression of the top eigenvalues of $H^{(n)}$ for $N-n \ll N^{2/3}$ in terms of averages of its eigenvectors and eigenvalues of its minor $H^{(n-1)}$ obtained by removing one row and one column. 
 Using the expression recursively, we can then write $N(\lambda_1^{(N)}-\lambda_1^{(N-k)})-k$ as a sum of $k$ martingale differences. Once this is done,  the desired result follows by a martingale CLT. To ensure this recursive argument can continue, we first need to prove that the level repulsion event for $\lambda_1^{(n)}-\lambda_2^{(n)}$ can be maintained during the iteration process from $n=N$ to $n=N-k$, as long as $k\ll N^{2/3}$. This also requires an iteration to prove, with the aid of the maximal inequality for martingale.  The same approach works to show the Gaussianity of other top eigenvalues $N(\lambda_i^{(N)}-\lambda_i^{(N-k)})-k$, for indices $i$ of order one. Nevertheless, our proof of the level repulsion across the iteration for index $i$ requires the level repulsion for indices smaller than $i$. This is the technical reason why we stop the extension at a fixed $i$. We then make one step further by finding an explicit expression for the eigenvectors ${\bm w}_i^{(n)}$ of $H^{(n)}$ and show that in the regime $N-n \ll N^{2/3}$ they are asymptotically aligned, i.e. $\langle {\bm w}_i^{(n)}, {\bm w}_i^{(n-1)}\rangle=1+O(N^{-2/3})$. Iterating this $k$ times and using that $k \ll N^{2/3}$, we conclude $\langle {\bm w}_i^{(N)}, {\bm w}_i^{(N-k)}\rangle\approx 1$. 

In the critical regime $k\sim N^{2/3}$  we show that the joint distribution of the top eigenvalues of $H^{(N)}$ and its minor $H^{(N-k)}$ is universal (in both real symmetric and complex Hermitian symmetry class), and it is given by the Forrester--Nagao distribution \cite{determinantal_corr} in the complex Hermitian case (see Theorem~\ref{theo:crit_k} for the explicit distribution). To do this we use a fairly standard long time Green's function comparison theorem (GFT), i.e.~we embed our Wigner matrix $H$ into an Ornstein--Uhlenbeck flow, and show that along this flow the joint statistics of the eigenvalues of $H^{(N)}$ and its minor $H^{(N-k)}$ are (asymptotically) unchanged.

Lastly, we discuss the more substantial supercritical regime $k\gg N^{2/3}$. In this regime the main ingredient is the almost orthogonality of the eigenvectors ${\bm w}_i^{(N)}$ and ${\bm w}_i^{(N-k)}$ for $k\gg N^{2/3}$. More precisely, we first show that
\begin{equation}
\label{eq:ovbound}
\big|\langle {\bm w}_i^{(N)}, {\bm w}_i^{(N-k)}\rangle\big|\ll 1,
\end{equation}
by relying on (an extension of) the recent two-resolvents local laws \cite[Theorem 3.2]{eigenv_decorr}, and then we use the \emph{Dyson Brownian Motion (DBM)} to show that this implies the independence of the eigenvalues of $H^{(N)}$ and its minor. We now explain in more detail both parts. 

We notice that to give an upper bound on $\big|\langle {\bm w}_i^{(N)}, {\bm w}_i^{(N-k)}\rangle\big|$ it is in fact enough to bound\footnote{ Here, we view $H^{(N-k)}$ to be an $N\times N$ matrix by adding $k$ rows and columns consisting of all zeros.}
\begin{equation}
\label{eq:2Glaw}
\langle \Im (H^{(N)}-z_1)^{-1}\Im (H^{(N-k)}-z_2)^{-1}\rangle\ll \frac{1}{|\Im z_1|}\wedge \frac{1}{|\Im z_2|},
\end{equation}
for appropriately chosen spectral parameters $z_1,z_2\in\C\setminus\R$ with imaginary part proportional to the local spacing of the eigenvalues. Objects as in \eqref{eq:2Glaw} have been largely studied in random matrix theory in the last few years under the name of \emph{multi--resolvent local laws}; see for example \cite{macroCLT_complex, Multi_res_llaws, cipolloni2023eigenstate, Cipolloni_meso, ding2024eigenvector, lin2024eigenvector}. To prove \eqref{eq:2Glaw} we rely on the \emph{Zigzag strategy} (see \cite{cipolloni2023eigenstate, Cipolloni_meso}) which consists of first proving \eqref{eq:2Glaw} for matrices with a fairly large Gaussian component using the \emph{method of characteristics} (i.e. by studying the evolution of the resolvent along a carefully chosen flow), and then remove this Gaussian component by an additional GFT argument. We point out that the \emph{method of characteristics} in the context of studying local properties of resolvents of random matrices was first used in \cite{Lee_edge} to prove edge universality of deformed Wigner matrices. Later, this idea was used to prove eigenvalue rigidity and universality in \cite{HL_rig, Bou_extreme}. Lastly, this approach was extended to study products of resolvents only more recently in \cite{bourgade2022liouville, Cipolloni_meso} and in subsequent papers. Once the bound \eqref{eq:ovbound} is proven, we can readily conclude the independence of the eigenvalues of $H^{(N)}$ and its minor by analyzing weakly correlated DBMs following the approach used in \cite{macroCLT_complex}.

\subsection{Organization of the paper}  We now briefly discuss how the rest of the paper is organized. In Section~\ref{sec:mainres} we state our main results. Then, in Section~\ref{sec:supercrit} we present the main steps for the proof of the asymptotic independence in the supercritical regime. Several technical results used in this section are proven in Sections~\ref{sec:overlap}--\ref{sec:GFT}. In particular, in Section~\ref{sec:GFT} we present a Green's function comparison (GFT) argument to show the density of matrices with a fairly large Gaussian component in the class of Wigner matrices. This GFT argument is then also used to show universality in the critical regime. Finally, in Section~\ref{sec:subcrit} we prove the strong dependence of the edge eigenvalues in the subcritical regime, using a martingale Central Limit Theorem. Several auxiliary results and proofs are deferred to two appendices.

\subsection*{Notations}
For positive quantities $f,g$ we write $f\lesssim g$ (or $f=\mathcal{O}(g)$) and $f\sim g$ if $f \le C g$ or $c g\le f\le Cg$, respectively, for some constants $c,C>0$ which only depend on the constants appearing in the moment condition (see Assumption \ref{ass:momass}). We also frequently use the notation $f \ll g$, which means that $f/g \to 0$ as $N\to \infty$.

For any natural number $n$ we set $[n]: =\{ 1, 2,\ldots ,n\}$. Matrix entries are indexed by lowercase Roman letters $a, b, c, ...$ from the beginning of the alphabet. We denote vectors by bold-faced lowercase Roman letters ${\bm x}, {\bm y}\in\C ^N$, or lower case Greek letters $\psi, \phi \in \C^N$, for some $N\in\N$. Vector and matrix norms, $\lVert {\bm x}\rVert$ and $\lVert A\rVert$, indicate the usual Euclidean norm and the corresponding induced matrix norm. For any $N\times N$ matrix $A$ we use the notation $\langle A\rangle:= N^{-1}\mathrm{Tr}  A$ for its normalized trace. Moreover, for vectors ${\bm x}, {\bm y}\in\C^N$ we denote their scalar product by $\langle {\bm x},{\bm y}\rangle:= \sum_{i} \overline{x}_i y_i$. 

Finally, we use the concept of ``with very high probability'' \emph{(w.v.h.p.)} meaning that for any fixed $C>0$, the probability of an $N$-dependent event is bigger than $1-N^{-C}$ for $N\ge N_0(C)$. Moreover, even if not explicitly stated, every estimate involving $N$ is understood to hold for $N$ being sufficiently large. \nc  We also introduce the notion of \emph{stochastic domination} (see e.g.~\cite{loc_sc_gen}): given two families of non-negative random variables
\[
X=\left(X^{(N)}(u) : N\in\N, u\in U^{(N)} \right) \quad \mathrm{and}\quad Y=\left(Y^{(N)}(u) : N\in\N, u\in U^{(N)} \right)
\] 
indexed by $N$ (and possibly some parameter $u$  in some parameter space $U^{(N)}$), 
we say that $X$ is stochastically dominated by $Y$, if for all $\xi, C>0$ we have 
\begin{equation}
	\label{stochdom}
	\sup_{u\in U^{(N)}} \mathbf{P}\left[X^{(N)}(u)>N^\xi  Y^{(N)}(u)\right]\le N^{-C}
\end{equation}
for large enough $N\ge N_0(\xi,C)$. In this case we use the notation $X\prec Y$ or $X= \mathcal{O}_\prec(Y)$.

\section{Main results}
\label{sec:mainres}

Let $H=(h_{ij})_{i,j=1}^N$ be either a real symmetric or a complex Hermitian Wigner matrix, i.e. the entries of $H$ are independent (up to the symmetry constraint) having distribution
\begin{equation*}
h_{ii}\stackrel{d}{=} \frac{1}{\sqrt{N}}\chi_\dif, \qquad h_{ij}\stackrel{d}{=}\frac{1}{\sqrt{N}}\chi_{{\rm od}},\quad i>j.
\end{equation*} 
On the ($N$-independent)
random variables $\chi_\dif \in \R$, $\chi_{\mathrm{od}}\in \C$ we formulate the following assumptions.
\begin{assumption}
\label{ass:momass}
Both $\chi_\dif$, $\chi_{\mathrm{od}}$ are centered $\E\chi_\dif=\E \chi_{\mathrm{od}}=0$ and have unit variance 
$\E \chi_\dif^2=  \E|\chi_{\mathrm{od}}|^2=1$. In the complex case we also assume that $\E\chi_{\mathrm{od}}^2=0$. Furthermore, we assume the existence of high moments, i.e. for any $p\in\N$ there exists a constant $C_p>0$ such that
\begin{equation}
\label{eq:momass}
\E\big[|\chi_\dif|^p+|\chi_{\mathrm{od}}|^p\big]\le C_p.
\end{equation}
\end{assumption}
For $k\in [N]$, let $H^{[k]}$ be the minor of $H$ formed by first $k$ rows and first $k$ columns. Similarly, let $H^{(N-k)}$ denote the minor formed by the last $N-k$ rows and the last $N-k$ columns, i.e
\begin{equation} \label{eq:minordef}
H = \begin{pmatrix}
H^{[k]} & X_k^*\\
X_k& H^{(N-k)}
\end{pmatrix}.
\end{equation} 
Here $X_k$ is a block of size $(N-k)\times k$. Furthermore, we denote the eigenvalues of $H^{(N-k)}$ by $\lbrace\lambda_j^{(N-k)}\rbrace_{j=1}^{N-k}$ in decreasing order and the corresponding $\ell^2$-normalized eigenvectors are denoted by $\bw_j^{(N-k)}\in\C^{N-k}$. With a slight abuse of notation, we will identify 
\begin{equation} \label{eq:identify}
\C^{N-k}\ni  \bm w_j^{(N-k)} \quad \leftrightarrow \quad  (\bm 0, \bm w_j^{(N-k)}) \in \C^N \,, 
\end{equation}
where we augmented the original vector with $k$ zeros from the top to form a vector in $\C^N$, and denote $(\bm 0, \bm w_j^{(N-k)})$ by $\bm w_j^{(N-k)}$ as well.  For $k=0$ we sometimes denote $H$ by $H^{(N)}$.

Returning to the eigenvalues, note that $\lambda_1^{(N)}$ is close to $2$, while $\lambda_1^{(N-k)}$ is close to $2(N-k)^{1/2}N^{-1/2}$, both with a fluctuation given by the celebrated Tracy-Widom law \cite{TW_level-spacing, TW_Orth_ME}, on the scale $N^{-2/3}$ and $(N-k)^{-2/3}$ respectively. We will consider the regime where $k\ll N$, and thus the fluctuation scale is $N^{-2/3}$ to leading order in both cases. While the distributions of  $\lambda_1^{(N)}$ and $\lambda_1^{(N-k)}$ separately are well known, we now study their correlations, i.e. joint distribution. Moreover, we consider the overlap $\big\langle \bm w_1^{(N)}, \bm w_1^{(N-k)} \big\rangle $ of the top eigenvectors of the original matrix and its minor.

\subsection{Main results: Decorrelation transition}
\label{sec:FNphaset} We study three different regimes: 
\begin{itemize}
\item[(i)] the \emph{supercritical regime} with $k\gg N^{2/3}$, cf.~Theorem \ref{theo:large_k} in Section \ref{subsec:supercrit}; 
\item[(ii)] the \emph{critical regime} with $k \sim N^{2/3}$, cf.~Theorem \ref{theo:crit_k} in Section \ref{subsec:crit};
\item[(iii)] the \emph{subcritical regime} with $k \ll N^{2/3}$, cf.~Theorem \ref{theo:small_k} in Section \ref{subsec:crit}.
\end{itemize}

\subsubsection{Supercritical regime} \label{subsec:supercrit}We begin with the supercritical regime, where we prove that (a) the top eigenvalues $\lambda_1^{(N)}$ and $\lambda_1^{(N-k)}$, 
appropriately rescaled,  become asymptotically independent in the large $N$ limit, and (b) the top eigenvectors become asymptotically orthogonal in the large $N$ limit. This is the statement of the following theorem.
\begin{theorem}[Supercritical regime: Asymptotic independence]\label{theo:large_k} Fix a (small) constant $\varepsilon_0>0$. Let $H$ be a Wigner matrix satisfying Assumption \ref{ass:momass} and let $k\in\N$ be such that $N^{2/3+\varepsilon_0}\le k\le N^{1-\varepsilon_0}$. Then we have the following: 
	\begin{itemize}
\item[(a)] \textnormal{[Eigenvalues]}	For any smooth compactly supported test-functions $F,G:\R\to\R$
it holds that
\begin{equation}
	\begin{split}
		&\E \left[F\big((\lambda_1^{(N)}-2)N^{2/3}\big) G\big((\lambda_1^{(N-k)}-2(1-k/N)^{1/2})(N-k)^{2/3}\big)\right]\\
		&\quad = \E \left[F\big((\lambda_1^{(N)}-2)N^{2/3}\big)\right] \E \left[G\big((\lambda_1^{(N-k)}-2(1-k/N)^{1/2})(N-k)^{2/3}\big)\right] + \mathcal{O}\left(N^{-c}\right)
	\end{split}
	\label{eq:main}
\end{equation}
for some constant $c>0$ which does not depend on $N$.
\item[(b)] \textnormal{[Eigenvectors]} Using the identification \eqref{eq:identify}, it holds that 
\begin{equation}
\big| \big\langle \bm w_1^{(N)}, \bm w_1^{(N-k)}\big\rangle \big|^2 \prec \frac{N^{2/3}}{k} \le N^{-\varepsilon_0} \,, 
\end{equation}
i.e.~$\bm w_1^{(N)}$ and $ \bm w_1^{(N-k)}$ are almost orthogonal. 
	\end{itemize}
\end{theorem}

We formulated Theorem \ref{theo:large_k}~(a) only for the top eigenvalue, but our proof also gives a similar result for the joint distribution of any $p$-tuple of top eigenvalues, i.e.~the collection of  appropriately
rescaled $p$-tuples $( \lambda_{i_1}^{(N)},   \ldots \lambda_{i_p}^{(N)})$
and $( \lambda_{j_1}^{(N-k)},   \ldots \lambda_{j_p}^{(N-k)})$
are asymptotically independent if $i_1,  \ldots i_p, j_1,  \ldots, j_p\le N^\varepsilon$. Similarly, also the top eigenvectors considered in Theorem \ref{theo:large_k}~(b) are asymptotically orthogonal, i.e.~$\big| \big\langle \bm w_j^{(N-k)}, \bm w_i^{(N)}\big\rangle\big| \ll 1$ for $i,j \le N^\varepsilon$. 

\subsubsection{Critical regime} \label{subsec:crit} Next, we turn to the critical regime, where we prove that the joint 
one-point correlation function  of $H^{(N)}$ and $H^{(N-k)}$ near the edge is universal. 
We expect that the corresponding overlap $\big| \big\langle \bm w_1^{(N)}, \bm w_1^{(N-k)}\big\rangle \big|^2$
also has a nontrivial and universal distribution.  However, since the limiting distribution 
is not known, we do not pursue this question further   in this paper. \nc

\begin{theorem}[Critical regime: Nontrivial joint distribution.]\label{theo:crit_k} Let $H$ be a complex Hermitian Wigner matrix satisfying Assumption \ref{ass:momass}. Fix a constant $\alpha>0$ and set $k:=\lfloor\alpha N^{2/3}\rfloor$.
Denote the joint one-point correlation function 
 of the eigenvalues $\{ \lambda_j^{(N)}\}_{j=1}^N$ of $H=H^{(N)}$ and $\{ \lambda_j^{(N-k)}\}_{j=1}^{N-k}$ of 
 $H^{(N-k)}$ by $p_\alpha^{(N)}:\R^2\to\R$, i.e.~it is  defined by the identity 
 $$
   \E \frac{1}{N(N-k)} \sum_{i,j} f(\lambda_i^{(N)}, \lambda_j^{(N-k)}) = \int_{\R}\int_\R p_\alpha^{(N)}\left(x,y\right) f(x,y) \dif x \dif y
$$ 
for any smooth compactly supported function $f$. 
For any parameters $X, Y\in\R$ introduce the edge-scaling  
\begin{equation*}
x:= 2+N^{-2/3}X,\quad y:= 2(1-k/N)^{1/2} + N^{-1/2} (N-k)^{-1/6} Y.
\end{equation*}
Then, in the sense of weak convergence, we have
\begin{equation*}
\lim_{N\to\infty} N^{2/3} p_\alpha^{(N)}\left(x,y\right) =4\det
\begin{pmatrix}
 \int_0^{+\infty} \mathrm{Ai}(X+u)\mathrm{Ai}(X+u)\dif u& \!\!\!\!\!\!\int_0^{+\infty} \ee^{-\alpha u} \mathrm{Ai}(X+u)\mathrm{Ai}(Y+u)\dif u\\
 &\\
 -\int_{-\infty}^{0} \ee^{\alpha u} \mathrm{Ai}(X+u)\mathrm{Ai}(Y+u)\dif u& \!\!\!\!\!\!\int_0^{+\infty} \mathrm{Ai}(Y+u)\mathrm{Ai}(Y+u)\dif u
\end{pmatrix},
\end{equation*}
where $\mathrm{Ai}$ is the Airy function.
\end{theorem}
The limiting distribution in Theorem~\ref{theo:crit_k} was explicitly computed by Forrester and Nagao for GUE matrices \cite[Eq. (5.35)]{determinantal_corr}. We also point out that \nc Theorem \ref{theo:crit_k} is stated specifically for the complex Hermitian case. This choice stems from our proof strategy which relies on comparing $H$ with a GUE matrix and applying the version of Theorem \ref{theo:crit_k} established for GUE in \cite[Proposition 5]{determinantal_corr}. However, our method extends to the real symmetric case. In fact, we prove that the joint distribution is universal in \emph{both} symmetry classes, just the explicit formula analogous to \cite[Eq. (5.35)]{determinantal_corr} is not (yet) available for GOE. Moreover, our proof shows that universality also holds for the $r$-point correlation functions at the edge regime
%of (some of) the top  $N^\varepsilon$ eigenvalues 
(for any finite $r$) even though we wrote up only the $r=1$ case. In the GUE case the process is determinantal,
so an explicit formula for the $r$-point function is  available, see  \cite[Eq. (5.35)]{determinantal_corr}. \nc

\subsubsection{Subcritical regime} \label{subsec:subcrit}Finally, we turn to the subcritical regime, where we prove that the \emph{difference} $\lambda_1^{(N)} - \lambda_1^{(N-k)}$ is approximately Gaussian with variance $\sqrt{k}/N \ll N^{-2/3}$, i.e.~much smaller than their individual fluctuation scale. This indicates the strong correlation of $\lambda_1^{(N)}$ and $\lambda_1^{(N-k)}$. Moreover, we show that the top eigenvectors are almost aligned, $\big|\big\langle \bm w_1^{(N)}, \bm w_1^{(N-k)}\big\rangle\big| \approx 1$. 

\begin{theorem}[Subcritical regime: Normal fluctuations and eigenvector alignment]\label{theo:small_k} Fix a (small) constant $\varepsilon_0>0$. Let $H$ be a Wigner matrix satisfying Assumption \ref{ass:momass}.
Then we have the following: 
\begin{itemize}
\item[(a)] \textnormal{[Eigenvalues]}  For $1\ll k\le N^{2/3-\varepsilon_0}$, \nc the difference $\lambda_1^{(N)} - \lambda_1^{(N-k)}$ is asymptotically normal\footnote{If $k = k(N) \le N^{2/3-\varepsilon_0}$ is such that $k(N)\to \infty$,  then \eqref{eq:normfluc} means that the left-hand-side~of \eqref{eq:normfluc} converges to the normal distribution $\mathcal{N}(0, 2/\beta)$ weakly, as $N \to \infty$.\label{ftn:asymptnorm}} fluctuating on scale $\sqrt{k}/N$,~i.e.
\begin{equation} \label{eq:normfluc}
\frac{N\big( \lambda_1^{(N)} - \lambda_1^{(N-k)}\big)-k}{\sqrt{k}} \sim \mathcal{N}(0,2/\beta) \,, 
\end{equation}
where $\beta=1$ in the real symmetric and $\beta=2$ in the complex Hermitian case and  $\mathcal{N}(0,2/\beta)$ is a real centered Gaussian random variable of variance $2/\beta$. 
\item[(b)] \textnormal{[Eigenvectors]}  Recalling the identification \eqref{eq:identify}, for any small  $\delta>0$ 
and any $1\le k\le N^{2/3-\varepsilon_0}$, \nc there exists a positive constant $a>0$, such that
\begin{equation*}
	\mathbf{P} \Big[0\leq 1-|\langle \bm{w}_1^{(N)}, {\bm{w}}_1^{(N-k)}\rangle|\leq k N^{-\frac23+\delta}\Big]\geq 1-N^{-a} \,, 
\end{equation*}
for $N$ large enough, i.e.~$\bm{w}_1^{(N)}$ and ${\bm{w}}_1^{(N-k)}$ are asymptotically almost aligned. 
\end{itemize}
\end{theorem}

We formulated Theorem \ref{theo:small_k} only for the top eigenvalue and top eigenvector, but our proof in Section~\ref{sec:subcrit} (see Propositions \ref{EigenvalueSub}--\ref{EigenvectorSub}) gives analogous results when replacing the index $i=1$ of the $\lambda$'s and $\bm w$'s by some $i \le C$ for an $N$-independent constant $C > 0$. 
 We can even show that the  $(\lambda_i^{(N)} - \lambda_i^{(N-k)})$'s are asymptotically independent.
 
 \begin{remark}

 In our statement of the normal fluctuation \eqref{eq:normfluc} we subtracted $k$ from 
 $N(\lambda_1^{(N)} - \lambda_1^{(N-k)})$ and not exactly its  expectation. The discrepancy is minor, in fact we expect that
 \begin{equation}
 \label{eq:k+lower}
 N\E(\lambda_1^{(N)} - \lambda_1^{(N-k)})=k+o(\sqrt{k})
 \end{equation}
 holds for any $1\ll k \le N^{2/3-\varepsilon_0}$, but strictly speaking this does not follow from existing results in
 the literature. 

In fact, in the regime  $k\ge N^{\epsilon_0}$,  \cite[Theorem~1.3]{schnelli2022convergence} 
proves, for any $\omega>0$,  an optimal $N^{-1/3+\omega}$ speed of convergence of the distribution  function  $\mathbf{P}( N^{2/3}(\lambda^{(N)}_1-2)\ge r)$ 
to the Tracy-Widom distribution.
This result is formulated for  any fixed $r$, which does not directly imply the same 
speed of convergence for the expectation. However, we believe that this theorem can easily  be extended to any
$r\ge -C(\log N)^{1/3}$, for some large $C>0$. This extension,  together with standard moderate deviation bounds on the tail
would imply that 
\[
 \E N^{2/3}(\lambda_1^{(N)} - 2 ) = \kappa + O(N^{-1/3 + \omega}),
\]
for some $\kappa \in \R$ (in fact $\kappa$ is the first moment of the $TW_\beta$ distribution). By simple
scaling, a similar asymptotic holds for the largest eigenvalue of the minor
\[
 \E (N-k)^{2/3}(\lambda_1^{(N-k)} - 2 \sqrt{1-k/N} ) = \kappa + O(N^{-1/3 + \omega}).
 \]
From these relations, simple algebra shows \eqref{eq:k+lower}.
 \end{remark}

\subsection{Extension towards the bulk}  
In Section~\ref{sec:FNphaset} we considered the decorrelation transition 
strictly in the edge regime. More  precisely, we studied the few top eigenvalues $\lambda_i$, with $i\sim 1$,
  of a Wigner matrix $H^{(N)}$ and its minor $H^{(N-k)}$, obtained from $H^{(N)}$ by removing the first $k$ rows and columns. We proved that their correlation
exhibits a phase transition when $k\sim N^{2/3}$; for $k\ll N^{2/3}$ the largest eigenvalues of $H^{(N)}$ and $H^{(N-k)}$ are strongly correlated, while in the complementary regime $k\gg N^{2/3}$ they are asymptotically independent.

While in Section~\ref{sec:FNphaset} we studied only the regime $i\sim 1$, we expect that a similar phase transition 
occurs  in the entire complement of the bulk of the spectrum. 
We claim that for any exponent $a\in [0,1)$ and eigenvalue index $i\sim N^a$ we have a transition (from strong dependence to independence) on a scale $k\sim N^{2(1-a)/3}$. The situation in the bulk, $a=1$,  is very different and will be explained
separately at the end (see also Figure~\ref{fig:general_a}).

\begin{figure}[h]
\begin{center}
\begin{subfigure}{0.4\textwidth}
\includegraphics[height=4.5cm]{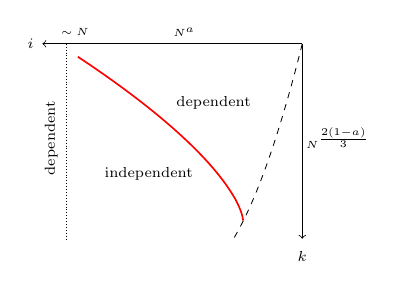}
\caption{}
\end{subfigure}%
~
\begin{subfigure}{0.6\textwidth}
\includegraphics[height=2.5cm]{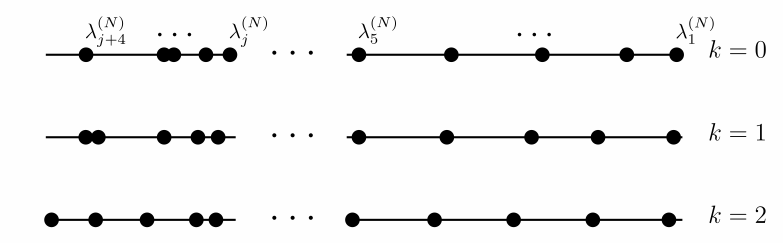}

\bigskip\bigskip\bigskip
\caption{}
\end{subfigure}
\end{center}
\caption{(A): Decorrelation transition in the intermediate regime $a\in [0,1)$. The figure depicts the shrinking spectrum of the minors $H^{(N-k)}$; the dashed line their largest eigenvalue. We consider $a\in[0,1)$, fix an index $i\sim N^a$ and track the trajectory of $\lambda^{(N-k)}_i$ as $k$ decreases. This trajectory crosses the solid red line at the level $k(a)\sim N^{2(1-a)/3}$, marking the phase transition: for $k\ll k(a)$, $\lambda^{(N-k)}_i$ and $\lambda^{(N)}_i$ remain highly correlated, while for $k\gg k(a)$, $\lambda^{(N-k)}_i$ becomes essentially independent of $\lambda^{(N)}_i$. For $a=1$, i.e. 
to the left of the dotted line (bulk regime), the eigenvalues remain correlated for any $k$.  \\[1mm]
	(B): Individual  eigenvalues of $H^{(N)}$ and its first two minors in the extreme edge regime and in the bulk regime.
	In the edge regime, the $i$-th eigenvalues $\lambda_i^{(N)},\lambda_i^{(N-1)}, \lambda_i^{(N-2)}\ldots $
	stick to each other indicating very strong correlation. In the bulk regime the only correlation present is
	the one enforced by interlacing.	 }
\label{fig:general_a}
\end{figure}

We now explain how to obtain the threshold  $k\sim N^{2(1-a)/3}$  for $a\in [0,1)$.
 We do not present the mathematical details for brevity, but just explain the general strategy. We separate the 
argument into two parts: we first show that the eigenvalues corresponding to indices $i\sim N^a$
are (asymptotically) independent for $k\gg N^{2(1-a)/3}$ and then that they are strongly correlated for
$k\ll N^{2(1-a)/3}$.

For $k\gg N^{2(1-a)/3}$, we follow a strategy similar to Section~\ref{sec:supercrit}, i.e. we show the independence of the eigenvalues using a  Dyson Brownian motion (DBM) argument. More precisely, we couple the 
time evolution of $\lambda_i^{(N)}(t)$ and $\lambda_i^{(N-k)}(t)$ under a Brownian matrix flow
with two fully independent DBM flows $\mu_i^{(1)}(t), \mu_i^{(2)}(t)$ showing that for larger times $t$
they are close to each other. To control the coupling, technically this requires 
 following analogue of the eigenvector overlap bound (cf. \eqref{eq:overlap_bound}):
\begin{equation}
\label{eq:newovb}
\big|\langle {\bm w}_i^{(N)},{\bm w}_j^{(N-k)}\rangle\big|^2\prec \frac{N^{2(1-a)/3}}{k} \wedge 1, \qquad\quad i,j\sim N^a.
\end{equation}
The proof of \eqref{eq:newovb} is very similar  to the one in Section~\ref{sec:ovb} below, once it is taken into consideration that the local eigenvalue density is of order $\sim N^{(1-a)/3}$, when $i\sim N^a$.
The bound \eqref{eq:newovb} shows that the eigenvectors of $H^{(N)}$ and $H^{(N-k)}$ are asymptotically orthogonal  $|\langle {\bm w}_i^{(N)},{\bm w}_j^{(N-k)}\rangle|\ll 1$ for $k\gg N^{2(1-a)/3}$. Armed with this bound as an input, we can then proceed with a DBM argument similar to the one in Section~\ref{sec:wcDBM}. The only difference is that in the current case we cannot directly rely on the results from \cite{Bou_extreme, landon2017edge}, which are formulated only in the regime $i\sim 1$. However, inspecting their proof, it is clear that similar arguments can be extended to the regime $i\sim N^a$, for $a\in (0,1)$. This gives the desired independence when $k\gg N^{2(1-a)/3}$.

For $k\ll N^{2(1-a)/3}$, we can follow a strategy similar to Section \ref{sec:subcrit}. As the level repulsion scale for $\lambda_i^{(n)}-\lambda_{i+1}^{(n)}$, for $n\in [N-k, N]$,  should now be changed to $n^{-(2+a)/3-\delta}$, the error term in each step of the recursion analogous to \eqref{19071007} will be $N^{(a-1)/3+C\delta}$ instead of $N^{-1/3+C\delta}$. Hence, after $k$ steps of the iteration, the error term will be of order $k N^{(a-1)/3+C\delta}$, as an extension of the error in \eqref{100730}. Since we will further renormalize $N(\lambda_i^{(N)}-\lambda_i^{(N-k)})$ by $\sqrt{k}$ in order to get the CLT, we require $\sqrt{k}N^{(a-1)/3+C\delta}$ to be negligible. This leads to the threshold $N^{2(1-a)/3}$ from the subcritical side. It is nevertheless unclear at this moment how to extend our argument on level repulsion across the iteration from a fixed
$N$-independent $i$ to $i\sim N^{a}$. 
Together with the explanation in the previous paragraph, this
  shows a transition at a scale $k\sim N^{2(1-a)/3}$, with $a\in [0,1)$.

One might naively think that this phase transition behavior extends also to the purely bulk regime.
However, the situation in the regime $a=1$ is very different. In fact, one can see that when $i\sim N$ 
\emph{single eigenvalues} are non-trivially correlated for any $k\sim N^b$, with $b\in [0,1)$, while \emph{eigenvalue gaps} are asymptotically independent for any  $k\gg 1$. The independence of the gaps follows analogously to the DBM--argument above (see also \cite[Section 4]{cipolloni2023quenched} and \cite[Remark 2.8]{eigenv_decorr}).  
Single eigenvalues are subject to interlacing which, for example,  restricts $\lambda_i^{(N-1)}$ to be in the $1/N$-vicinity
of $\lambda_i^{(N)}$ for $i\sim N$. Since the individual eigenvalues fluctuate on the bigger $\sqrt{\log N}/N$
scale (see \cite[Theorem~1.6]{BMeig}, \cite[Corollary 1.9]{bourgade2022optimal}, \cite[Corollary 1.3]{mody2023log}, and \cite[Theorem 1.3]{landon2020applications}), this restriction is severe, leading to strong correlation. 
Surprisingly, this correlation does not disappear when $\lambda_i^{(N-k)}$ is compared with $\lambda_i^{(N)}$ for
large $k$; interlacing alone cannot explain this. \nc

To see that single eigenvalues are correlated up to $k\ll N$ one can use that by \cite[Theorem 3.1]{Bou_extreme} we have
\begin{equation}
\label{eq:homog}
\lambda_i^{(N)}(t)=\mu_i^{(1)}(t)+F_1(t,N), \qquad\quad \lambda_i^{(N-k)}(t)=\mu_i^{(2)}(t)+F_2(t,N-k), \qquad\quad i\sim N,
\end{equation}
for any $1/N\ll t \ll 1$.
Here  the random variable $F_1$ depends only on the initial conditions $\lambda_i^{(N)}(0), \mu_i^{(1)}(0)$ 
and it is such that $|F_1|\prec 1/N$, and similarly for $F_2$. 
In fact, these quantities are mesoscopic linear statistics of eigenvalues and they
represent  collective fluctuations of many eigenvalues effectively on an energy scale of order $t$. 
Here $\mu_i^{(1)}(t),\mu_i^{(2)}(t)$ are independent of each other for any $t\ge 0$. From \eqref{eq:homog} we thus show that the dependence between $\lambda_i^{(N)}(t), \lambda_i^{(N-k)}(t)$ lies purely in the dependence of the $F_i$ (we choose $t\sim N^{-1+\omega}$, for a small $\omega>0$). However, using resolvent methods,  one can prove that
\begin{equation}
\label{eq:logcorr}
\mathrm{Cov}\big[F_1,F_2]\sim \frac{\log(t+k/N)}{N^2}.
\end{equation}
Recall that the natural fluctuation scale of the bulk eigenvalues is of order $\sqrt{\log N}/N$.
Therefore  $\lambda_i^{(N)}(t)$ and $\lambda_i^{(N-k)}(t)$ would become independent 
only when $\mathrm{Cov}\big[F_1,F_2] \ll (\log N)/N^2$, i.e. when $b=1$  (in the scaling  $k\sim N^b$)
and they are genuinely correlated for any $b\in [0,1)$.
Strictly speaking \eqref{eq:logcorr} is proven only for order one test function (i.e. when $t\sim 1$) in \cite[Proposition 1]{Bor} (see also \cite[Proposition 1]{BorII}); however, we believe that using the method of characteristics for $F_i$ one can obtain \eqref{eq:logcorr} (see e.g. \cite[Section 5]{bourgade2024fluctuations}, \cite[Sections 5-6]{cipolloni2024maximum}, and \cite[Section 7]{landon2022almost}).

We point out that the mesoscopic terms $F_1, F_2$ in \eqref{eq:homog} are present in every regime not only in the bulk and their size is essentially of order $1/N$ everywhere (see, e.g., \cite[Theorem~2.3]{HeKnowles} restricted to the bulk, as well as \cite[Theorem 2.4]{cipolloni2023functional} and \cite[Theorems~2.10-2.11]{LSX21}  valid for the entire spectrum
including the relevant edge regimes but only  for compactly supported mesoscopic functions).  The main difference between bulk indices $i\sim N$ and indices  $i\sim N^a$, $a<1$, closer to the edge is that in the latter case the natural fluctuation scale of $\lambda_i$ is of order $N^{-(2+a)/3}$, i.e. away from the bulk it is
much bigger than the natural size of the order $1/N$ of the mesoscopic terms. Therefore the nontrivial correlation
between $F_1, F_2$ (still present away from the bulk) becomes negligible on the relevant fluctuation scale.

\section{Supercritical and critical regime: Proofs of Theorems \ref{theo:large_k}--\ref{theo:crit_k}}
\label{sec:supercrit}
In this section, we provide the proofs of Theorems \ref{theo:large_k}--\ref{theo:crit_k}. 
\subsection{Proof of Theorem \ref{theo:large_k}}
In this section we focus on the supercritical regime, where $k\gg N^{2/3}$. As the first step, we show that in this regime the eigenvectors of $H^{(N)}$ and $H^{(N-k)}$ become nearly orthogonal to each other, as stated in Proposition \ref{prop:overlap} below. The proof is given in Section \ref{sec:overlap}. 
\begin{proposition}[Step 1: Bound on the eigenvector overlaps]\label{prop:overlap} Fix (small) $\varepsilon_0>0$ and let $H$ be a Wigner matrix satisfying Assumption \ref{ass:momass}. 
%There exists a constant $0<\varepsilon_1<\varepsilon_0$ which depends only on $\varepsilon_0$ such that 
Uniformly in $k\le N^{1-\varepsilon_0}$ and in $i,j\le N^{\varepsilon_0}$ we have that
	\begin{equation}
		\big\vert \big\langle {\bm w}_i^{(N)}, {\bm w}_j^{(N-k)}\big\rangle\big\vert^2\prec \frac{N^{2/3}}{k}\wedge 1,
		\label{eq:overlap_bound}
	\end{equation}
	where the normalized eigenvectors of the minor, $\bw_j^{(N-k)}$, are extended to a vector with $N$ coordinates by taking first $k$ coordinates equal to zero.
\end{proposition}

This immediately proves Theorem \ref{theo:large_k}~(b) and, at the same time, serves as the basis for the following \emph{Dyson Brownian Motion (DBM) analysis} for the eigenvalues $\lambda_1^{(N)}$ and $\lambda_1^{(N-k)}$ in the second step. 

In order to formulate this second step precisely, consider the matrix flow   
\begin{equation}
	\dif H_t^{(N)} = \frac{\dif B_t^{(N)}}{\sqrt{N}},\quad H_0^{(N)}=H^{(N)},
	\label{eq:H_flow}
\end{equation}
where $B_t^{(N)}$ is either real symmetric or a complex Hermitian standard matrix-valued Brownian motion with variances equal to $Nt$ times the second moments of GOE/GUE. The flow \eqref{eq:H_flow} induces the following flow of $H^{(N-k)}$ from \eqref{eq:minordef}: 
\begin{equation}
	\dif H_t^{(N-k)} = \frac{\dif B_t^{(N-k)}}{\sqrt{N}},\quad H^{(N-k)}_0=H^{(N-k)},
	\label{eq:H_min_flow}
\end{equation}
where $B_t^{(N-k)}$ is an $(N-k)\times (N-k)$ matrix-valued Brownian motion obtained from $B_t^{(N)}$ by removing the first $k$ rows and $k$ columns.\footnote{Note that the normalization factor $N^{-1/2}$ in \eqref{eq:H_min_flow} slightly differs from the normalization factor $(N-k)^{-1/2}$ in a standard Brownian flow for $H^{(N-k)}$, however this does not cause any difficulties.} In order to unify notations we set $N_1:=N$ and $N_2:=N-k$ and, similarly to the definition below \eqref{eq:minordef}, denote the eigenvalues of $H^{(N_l)}_t$ in decreasing order by $\lbrace \lambda_i^{(l)}(t)\rbrace_{i=1}^{N_l}$ for $l=1,2$, and let $\lbrace{\bm w}^{(l)}_i(t)\rbrace_{i=1}^{N_l}$ be the corresponding $\ell^2$-normalized eigenvectors. It is well-known (see, e.g., \cite{DBM_init} or \cite[Chapter 4.3.1]{AGZ}) that the eigenvalues of $H^{(N_l)}_t$ follow the \emph{Dyson Brownian motion}
\begin{equation}
	\dif \lambda^{(l)}_i(t) = \frac{\dif b^{(l)}_i(t)}{\sqrt{\beta N}} + \frac{1}{N}\sum_{j\neq i} \frac{1}{\lambda^{(l)}_j(t) -\lambda^{(l)}_i(t)}\dif t,\quad  i\in [N_l]
	\label{eq:init_DBM}
\end{equation}
with $\beta=1$ in the real symmetric case and $\beta=2$ in the complex Hermitian. Here $\lbrace b^{(1)}_i(t)\rbrace_{i=1}^{N_1}$ and $\lbrace b^{(2)}_j(t)\rbrace_{j=1}^{N_2}$ are collections of independent standard Brownian motions with the joint correlation structure
\begin{equation}
	\dif \big[ b_i^{(l_1)}(t), b_j^{(l_2)}(t)\big] = \big\vert \big\langle {\bm w}^{(l_1)}_i(t), {\bm w}^{(l_2)}_j(t)\big\rangle\big\vert^2\dif t =:\Theta_{ij}^{12}(t)\dif t, \quad l_1, l_2\in [2],
	\label{eq:BM_corr}
\end{equation}
for all $i\in[N_{l_1}]$ and $j\in[N_{l_2}]$. The relation \eqref{eq:BM_corr} follows from
\[
\dif b_i^{(l)}(t)=\frac{[(U^{(l)})^*(\dif B_t)U^{(l)}]_{ii}+\overline{[(U^{(l)})^*(\dif B_t)U^{(l)}]_{ii}}}{\sqrt{2}},
\]
as can it be seen from the derivation of the Dyson Brownian motion in \cite[Section 12.2]{erdHos2017dynamical} (see also \cite[Appendix B]{macroCLT_complex}). Here $U^{(l)}=U^{(l)}(t)$ denotes the matrix of the eigenvectors of $H_t^{(N_l)}$, i.e. ${\bm w}_1^{(l)}$ is the first column of $U^{(l)}$, ${\bm w}_2^{(l)}$ is the second column, and so on. In particular, $\Theta^{11}_{ij}=\Theta^{22}_{ij}=\delta_{ij}$, but notice that off-diagonal correlations, $\Theta^{12}, \Theta^{21}$ are typically non-zero, but small in the supercritical regime. Indeed, the overlap bound from Proposition \ref{prop:overlap} shows that the driving Brownian motions in \eqref{eq:init_DBM} for $i \le N^{\varepsilon_1}$ are almost independent. This allows to infer that after sufficiently long time $t$, the top eigenvalues of $H_t^{(N_1)}$ and $H_t^{(N_2)}$ are essentially independent. 
\begin{proposition}[Step 2: Approximate independence for Gaussian divisible ensemble] \label{prop:DBMindep}
Let $H$ be a Wigner matrix satisfying Assumption \ref{ass:momass}. Fix a (small) $\varepsilon_0 > 0$. Then there exist constants $\varepsilon_1 > 0$ and $c > 0$ such that for $k \ge N^{2/3 + \varepsilon_0}$, $i,j \le N^{\varepsilon_1}$, $t \ge N^{-1/3 + 100c}$ and any smooth compactly supported test-functions $F,G:\R\to\R$ it holds that
\begin{equation}
	\begin{split}
		&\E \left[F\big((\lambda_i^{(1)}(t)-2)N_1^{2/3}\big) G\big((\lambda_j^{(2)}(t)-2(N_2/N_1)^{1/2})N_2^{2/3}\big)\right]\\
		&\quad = \E \left[F\big((\lambda_i^{(1)}(t)-2)N_1^{2/3}\big)\right] \E \left[G\big((\lambda_j^{(2)}(t)-2(N_2/N_1)^{1/2})N_2^{2/3}\big)\right] + \mathcal{O}\left(N^{-c}\right) \,, 
	\end{split}
	\label{eq:step2}
\end{equation}
where $\lambda_i^{(1)}(t)$ and $\lambda_j^{(2)}(t)$ are the solutions to \eqref{eq:init_DBM}. 

The same approximate factorization property as in \eqref{eq:step2} holds also for the joint distribution of any $p$-tuple of top-eigenvalues  i.e.~the collection of  appropriately
rescaled $p$-tuples $( \lambda_{i_1}^{(1)}(t),   \ldots \lambda_{i_p}^{(1)}(t))$
and $( \lambda_{j_1}^{(2)}(t),   \ldots \lambda_{j_p}^{(2)}(t))$
with $i_1,  \ldots i_p, j_1,  \ldots, j_p\le N^\varepsilon$, and test functions $F, G : \R^p \to \R$. 
\end{proposition}
The proof of Proposition \ref{prop:DBMindep} is given in Section \ref{sec:wcDBM} and it gives Theorem~\ref{theo:large_k}~(a) for ensembles with a Gaussian component of size $N^{-1/3+\varepsilon}$. As the final step for proving Theorem \ref{theo:large_k}~(a), we remove the Gaussian component introduced by \eqref{eq:init_DBM} via a \emph{Green function comparison} (GFT) argument, which is proven in Section \ref{sec:GFT}. 
\begin{proposition}[Step 3: Removing the Gaussian component] \label{prop:GFTeigenval}
Let $H$ be a Wigner matrix satisfying Assumption \ref{ass:momass}, fix a (small) $\varepsilon_0 > 0$ and take $k\ge N^{2/3 + \varepsilon_0}$. Assume that for some small constants $\varepsilon_1, c > 0$ and $i,j \le N^{\varepsilon_1}$, $t \ge N^{-1/3 + 100c}$, and any smooth compactly supported test-functions $F,G:\R\to\R$, we have that \eqref{eq:step2} holds. Then \eqref{eq:step2} holds for time $t=0$ as well. 

Moreover, the same conclusion holds also for the joint distribution of any $p$-tuple of top-eigenvalues  i.e.~the collection of  appropriately
rescaled $p$-tuples $( \lambda_{i_1}^{(1)}(t),   \ldots \lambda_{i_p}^{(1)}(t))$
and $( \lambda_{j_1}^{(2)}(t),   \ldots \lambda_{j_p}^{(2)}(t))$
with $i_1,  \ldots i_p, j_1,  \ldots, j_p\le N^\varepsilon$, and test functions $F, G : \R^p \to \R$. 
\end{proposition}
Combining Proposition \ref{prop:DBMindep} with Proposition \ref{prop:GFTeigenval}, we immediately conclude the proof of Theorem \ref{theo:large_k}~(a) and thus the whole proof of Theorem \ref{theo:large_k}. \qed

\subsection{Proof of Theorem \ref{theo:crit_k}}

The proof of Theorem \ref{theo:crit_k} entirely relies on the following comparison argument, proven in Section \ref{sec:GFT}. 

\begin{proposition}[Comparison with the Gaussian ensemble] \label{prop:GFTeigenval2}
Let $H$ be a real symmetric or complex Hermitian Wigner matrix satisfying Assumption \ref{ass:momass}. Fix a constant $\alpha > 0$ and set $k := \lfloor \alpha N^{2/3} \rfloor$. Moreover, fix a small $\varepsilon > 0$ and consider the joint $r$-point correlation functions of $H^{(N)}$ and $H^{(N-k)}$ in the edge regime and  denote it by $p_{r, \alpha}^{(N)} : \R^{2r} \to \R$. Denote the analog for the GOE/GUE (corresponding to the symmetry class of $H$) ensemble by $\widetilde{p}_{r, \alpha}^{(N)} : \R^{2r} \to \R$. 

Then, in the sense of weak convergence (i.e. tested against smooth compactly supported test functions), it holds that 
\begin{equation}
\lim\limits_{N \to \infty} N^{2r/3} \left(p_{r, \alpha}^{(N)}(x_1, ... , x_r, y_1, ... , y_r) - \widetilde{p}_{r, \alpha}^{(N)}(x_1, ... , x_r, y_1, ... , y_r)\right) = 0 \,. 
\end{equation}
\end{proposition}
Given Proposition \ref{prop:GFTeigenval2}, Theorem \ref{theo:crit_k} readily follows by invoking the convergence of $N^{2/3}\widetilde{p}_{1, \alpha}^{(N)}(x, y)$ in the complex Hermitian case. \qed

\section{Eigenvector overlaps: Proof of Proposition \ref{prop:overlap}}\label{sec:overlap}
In this section, we prove the overlap bound in Proposition \ref{prop:overlap} via a suitable two-resolvent local for the resolvents of $H^{(N)}$ and $H^{(N-k)}$. The key idea to relate the two resolvents with each other is to view the resolvent of $H^{(N)}$ as the resolvent of a \emph{deformation} of $H^{(N-k)}$ via \emph{Schur complement formula}. 

More precisely, taking $z_1, z_2\in \mathbf{H}$, we will extract the overlap bound \eqref{eq:overlap_bound} from an upper bound on $\left\vert\langle \Im (H^{(N)}-z_1)^{-1}\Im F(z_2)\rangle\right\vert$, where $F(z_2)$ is extended from $(H^{(N-k)}-z_2)^{-1} \in \C^{(N-k) \times (N-k)}$ to an $N\times N$ matrix by writing zeros in the first $k$ rows and columns. Then, using Schur complement formula for the minor of $(H^{(N)}-z_1)^{-1}$ formed by the last $(N-k)$ rows and columns we find that
\begin{equation}
\langle \Im (H-z_1)^{-1}\Im F(z_2)\rangle = \frac{1}{N}{\rm Tr}\left[ \Im \left(H^{(N-k)}-D_k - z_1\right)^{-1} \Im (H^{(N-k)}-z_2)^{-1} \right]
\label{eq:2G_goal}
\end{equation}
where the deformation $D_k$ is given by
\begin{equation}
D_k:=X_k(H^{[k]}-z_1)^{-1}X_k^*.
\label{eq:D_def}
\end{equation}

In order to estimate \eqref{eq:2G_goal}, we will first collect some preliminary information on the deformation $D_k$ in Section \ref{subsec:prelim} and afterwards, in Section \ref{sec:ovb}, show that $|\eqref{eq:2G_goal}| \prec N^{2/3}/k \wedge 1$, yielding the desired bound~\eqref{eq:overlap_bound}. 

\subsection{Preliminaries: Analysis of the deformation} \label{subsec:prelim}
To analyze $D_k$, we first establish the concentration of the sample covariance matrix $X_kX_k^*$ around the identity matrix in Lemma \ref{lem:XX^*}. We then outline several important properties of $D_k$ in Lemma \ref{lem:D}, which are essential for studying \eqref{eq:2G_goal}. The proofs of both lemmas are postponed to Section \ref{subsec:prelimproof}. 
\begin{lemma}[Concentration of $X_kX_k^*$ around the identity matrix]\label{lem:XX^*} Fix (small) $\varepsilon_0, c>0$. Then for any $k\le N^{1-\varepsilon_0}$ with very high probability it holds that
\begin{equation}
\lVert X^*_kX_k - I_k\rVert\le c,
\label{eq:XX_concentr}
\end{equation}
as $N$ goes to infinity. Here $I_k$ is a $k\times k$ identity matrix.
\end{lemma}
\begin{lemma}[Properties of $D_k$]\label{lem:D} Fix (small) $\varepsilon_0, \delta,c>0$. Consider $z_1=E_1+\ii\eta_1\in\mathbf{H}$ with $ \vert E_1\vert \ge \delta$. In the regime $ N^{\epsilon_0} \le k\le N^{1-\varepsilon_0}$ with very high probability we have the following bounds on $\Re D_k$ and $\Im D_k$:
\begin{equation}
\lVert \Re D_k\rVert \le \vert E_1\vert^{-1}(1+c) \quad\text{and}\quad 0\le\Im D_k\le \eta_1\vert E_1\vert^{-2}(1+c) I.
\label{eq:D_bounds}
\end{equation}
We also have that
\begin{equation}
\frac{1}{N-k}{\rm Tr}\left[\left(\Re D_k-\frac{1}{N-k}{\rm Tr}\,\Re D_k\right)^2  \right]\sim \frac{1}{N-k}{\rm Tr}\left[\left(\Re D_k\right)^2  \right]\sim \frac{k}{N}
\label{eq:D_HS}
\end{equation}
with implicit constants depending only on $\delta$.
\end{lemma}

Note that $H^{(N-k)}-D_k$ is not Hermitian since $\Im D_k\neq 0$. However, since $0 \le \Im D_k \lesssim \eta_1$, as shown in \eqref{eq:D_bounds}, we can remove $\Im D_k$ from the resolvent of $H^{(N-k)} - D_k$ without changing its size. This is the content of the following general lemma, whose proof is given in Section \ref{subsec:prelimproof}. 

\begin{lemma}[Removing matrix from the imaginary part]\label{lem:Im} Fix a (large) $L>0$. Let $H$ and $S$ be $N\times N$ Hermitian matrices. Assume that $S\ge 0$ and $\lVert S\rVert\le L$. Then for any $\eta>0$ it holds that
\begin{equation}
c\Im (H-\ii\eta)^{-1}\le \Im (H-\ii\eta(I+S))^{-1}\le C\Im (H-\ii\eta)^{-1}
\label{eq:Im}
\end{equation}
for some $c,C>0$ which depend only on $L$.
\end{lemma}

\subsection{Proof of Proposition \ref{prop:overlap} \label{sec:ovb} via a two-resolvent local law for deformed Wigner matrices}
Fix indices $i,j\le N^{\varepsilon_0}$. The following argument holds uniformly for all such $i,j$. First of all, let $\rho_{sc}(x)=\frac{1}{2\pi}\sqrt{[4-x^2]_+}$ be the semicircular distribution. We rescale it as follows
\begin{equation*}
\rho^{(N-k)}(x):=(1-k/N)^{-1/2}\rho_{sc}\left( (1-k/N)^{-1/2}x\right)
\end{equation*} 
and define the \emph{classical eigenvalue locations} by
\begin{equation}
	\int_{\gamma_j^{(N-k)}}^{+\infty}\rho^{(N-k)}(x)\dif x=\frac{j}{N-k},\quad j\in[N-k].
\label{eq:quantiles}
\end{equation}
Recall the definition of the matrix $F$ given above \eqref{eq:2G_goal} and denote $z_l:=E_l+\ii\eta_l$, $l=1,2$. Using the spectral decomposition of $(H-z_1)^{-1}$ and $F$ we get that
\begin{equation}
\langle\Im (H-z_1)^{-1}\Im F(z_2)  \rangle\ge \frac{1}{N}\cdot \frac{\eta_1}{(\lambda_i^{(N)}-E_1)^2+\eta_1^2}\cdot \frac{\eta_2}{(\lambda_j^{(N-k)}-E_2)^2+\eta_2^2}\big| \langle {\bm w}_i^{(N)}, {\bm w}_j^{(N-k)}\rangle\big|^2.
\label{eq:overlap_2G_bound}
\end{equation}
In \eqref{eq:overlap_2G_bound}, we now take $E_1:=\gamma_i^{(N)}$, $E_2:=\gamma_j^{(N-k)}$ and $\eta_1=\eta_2:=N^{-2/3+\xi_0}$ for some small fixed constant $\xi_0>0$, and arrive at
\begin{equation}
\big| \langle {\bm w}_i^{(N)}, {\bm w}_j^{(N-k)}\rangle\big|^2 \prec N^{-1/3} \langle\Im (H-z_1)^{-1}\Im F  \rangle.
\label{eq:overlap_2G_reduced}
\end{equation}
Here we additionally used the standard rigidity bound from \cite[Theorem 7.6]{loc_sc_gen} which uniformly in the spectrum asserts that
\begin{equation} \label{eq:rigid}
\vert\lambda_i^{(N)}-\gamma_i^{(N)}\vert\prec N^{-2/3},\quad \vert\lambda_j^{(N-k)}-\gamma_j^{(N-k)}\vert \prec N^{-2/3}.
\end{equation}
Next we rewrite the right-hand-side~of \eqref{eq:overlap_2G_reduced} using \eqref{eq:2G_goal} and remove $\Im D_k$ by the means of Lemma \ref{lem:Im} applied\footnote{We point out that Lemma \ref{lem:D} implies that $S=\eta_1^{-1}\Im D_k$ satisfies the conditions of Lemma \ref{lem:Im}.} to $S:=\eta_1^{-1}\Im D_k$:
\begin{equation}
\big| \langle {\bm w}_i^{(N)}, {\bm w}_j^{(N-k)}\rangle\big|^2\prec N^{-1/3}N^{-1}{\rm Tr}\left[ \Im \left(H^{(N-k)}-\Re D_k - z_1\right)^{-1} \Im (H^{(N-k)}-z_2)^{-1} \right].
\label{eq:ov<2G_fin}
\end{equation}
We claim that
\begin{equation}
\left\langle \Im \left(H^{(N-k)}-\Re D_k - z_1\right)^{-1} \Im (H^{(N-k)}-z_2)^{-1} \right\rangle\prec \left( \left\langle(\Re D_k - \left\langle\Re D_k\right\rangle)^2\right\rangle  \right)^{-1},
\label{eq:2G_special}
\end{equation}
where $\langle\cdot\rangle$ stands for the normalized trace of an $(N-k)\times (N-k)$ matrix, i.e. $\langle\cdot\rangle=(N-k)^{-1}\mathrm{Tr}[\cdot]$. The proof of \eqref{eq:2G_special} relies on the two-resolvent local law from Proposition \ref{prop:2G} and is provided in Appendix~\ref{subsec:2G_special}.
Finally, applying \eqref{eq:D_HS} to bound the right-hand-side of \eqref{eq:2G_special} from above and combining this with \eqref{eq:ov<2G_fin}, we complete the proof of Proposition \ref{prop:overlap}. \qed

\subsection{Proof of the lemmas from Section \ref{subsec:prelim}} \label{subsec:prelimproof}
We conclude Section \ref{sec:overlap} by proving Lemmas \ref{lem:XX^*} -- \ref{lem:Im}.

\begin{proof}[Proof of Lemma \ref{lem:XX^*}]
Fix a small constant $0<\zeta<\varepsilon_0/4$ and introduce the truncation
\begin{equation*}
	\widehat{h}_{ij}:=h_{ij}\bm{1}_{\vert\sqrt{N} h_{ij}\vert\le N^{\zeta}},\quad \widehat{X}_k:= \left( \widehat{h}_{ij}\right)_{\substack{k+1\le i\le N\\ 1\le j\le k}}\quad\text{and}\quad \sigma_N^2:=N\E \vert \widehat{h}_{ij}\vert^2= N\E\left\vert\chi_{\mathrm{od}}\bm{1}_{\chi_{\mathrm{od}\le N^\zeta}}\right\vert^2.
\end{equation*}
With this normalization $\sigma_N^2=1+o(1)$ for large $N$. Using \eqref{eq:momass} we see that
\begin{equation}
	\mathbf{P}\left[ X_k\neq \widehat{X}_k\right] \le C_p\frac{Nk}{N^{p\zeta}}
	\label{eq:cutoff}
\end{equation}
for any $p\in\N$, and the right-hand-side of this inequality can be made smaller than $N^{-D}$ for any fixed $D>0$ by choosing $p$ sufficiently large.  Applying \cite[Theorem 5.41]{sample_cov_classical} to the matrix $\frac{\sqrt{N}}{\sigma_N} \widehat{X}_k$ we get that
\begin{equation*}
	\sqrt{N-k}-t\sqrt{\frac{kN^{2\zeta}}{\sigma_N^2}} \le s_{\mathrm{min}}\left(\frac{\sqrt{N}}{\sigma_N} \widehat{X}_k\right)\le s_{\mathrm{max}}\left(\frac{\sqrt{N}}{\sigma_N} \widehat{X}_k\right)\le \sqrt{N-k}+t\sqrt{\frac{kN^{2\zeta}}{\sigma_N^2}} 
\end{equation*}
for any $t>0$ with probability at least $1-2k\mathrm{exp}\{-c_0t^2\}$, where $c_0>0$ is an absolute constant and $s_{\mathrm{min}}$, $s_{\mathrm{max}}$ are the lowest and the largest singular values of $\frac{\sqrt{N}}{\sigma_N} \widehat{X}_k$. Here we additionally used that the $\ell^2$ norms of the rows of $\frac{\sqrt{N}}{\sigma_N} \widehat{X}_k$ are upper bounded by $\sqrt{kN^{2\zeta}/\sigma_N^2}$ almost surely. Therefore,
\begin{equation}
	1+o(1)-tN^{-\varepsilon_0/2}\le s_{\mathrm{min}}\left(\widehat{X}_k\right)\le s_{\mathrm{max}}\left(\widehat{X}_k\right)\le 1+o(1)+tN^{-\varepsilon_0/2}
	\label{eq:sing_bound}
\end{equation}
with probability at least $1-2k\mathrm{exp}\{-c_0t^2\}$. In particular, \eqref{eq:sing_bound} applied for $t:=N^{\varepsilon_0/4}$ implies that $\lVert \widehat{X}^*_k\widehat{X}_k - I_k\rVert\le c$ with very high probability for sufficiently large $N$. Together with \eqref{eq:cutoff} this finishes the proof of Lemma \ref{lem:XX^*}.
\end{proof}

\begin{proof}[Proof of Lemma \ref{lem:D}]
We start with the proof of \eqref{eq:D_bounds}. We have
\begin{equation}
	\left\lVert D_k + \frac{1}{z_1}X_kX_k^*\right\rVert = \left\lVert X_k \frac{H^{[k]}}{z_1(H^{[k]}-z_1)} X_k^*\right\rVert \le \lVert X_k\rVert^2 \left\lVert \frac{H^{[k]}}{z_1(H^{[k]}-z_1)}\right\rVert
	\label{eq:D-XX}
\end{equation}
Since $\sqrt{N/k}H^{[k]}$ is a $k\times k$ Wigner matrix, it holds that $\lVert H^{[k]}\rVert\le 3\sqrt{k/N}$ with very high probability. Therefore,
\begin{equation}
	\left\lVert \frac{H^{[k]}}{z_1(H^{[k]}-z_1)}\right\rVert \le \frac{6}{\delta^2} \sqrt{\frac{k}{N}}\le N^{-\varepsilon_0/3}
	\label{eq:mid_G_bound}
\end{equation}
for sufficiently large $N$. In combination with the bound $\lVert X_k\rVert\le \lVert X_k^*X_k\rVert^{1/2}$ and with Lemma \ref{lem:XX^*} this shows that the left-hand-side of \eqref{eq:D-XX} is upper bounded by $N^{-\varepsilon_0/3}$. Using Lemma \ref{lem:XX^*} again we conclude
\begin{equation*}
	\lVert \Re D_k\rVert \le \vert \Re (z^{-1})\vert\cdot \lVert X_kX_k^*\rVert + N^{-\varepsilon_0/3} \le \vert E_1\vert^{-1}(1+c).
\end{equation*}
Next, in order to verify the inequality $\Im D_k>0$ it is sufficient to note that
\begin{equation*}
	\Im D_k = X_k\Im \left[ (H^{[k]}-z_1)^{-1} \right] X_k^*\ge 0.
\end{equation*} 
The proof of the rest of \eqref{eq:D_bounds} is analogous to the argument presented above and thus is omitted.

Now we prove \eqref{eq:D_HS}. Denote for short $G^{[k]}:=(H^{[k]}-z_1)^{-1}$. Using \eqref{eq:mid_G_bound} and Lemma \ref{lem:XX^*} we get
\begin{equation*}
	\begin{split}
		{\rm Tr}\, \Re D_k = &-\Re(z_1^{-1}) {\rm Tr}\left[ X_kX_k^*\right] +  \mathrm{Tr} \left[ \Re\left(G^{[k]}+z_1^{-1}\right)X_k^*X_k\right]\\
		= & -\Re(z_1^{-1}) {\rm Tr}\left[ X_kX_k^*\right] + \mathcal{O}\left( k\left\lVert \Re\left(G^{[k]}+z_1^{-1}\right)X_k^*X_k \right\rVert\right)\sim k.
	\end{split}
\end{equation*}
Similarly we have ${\rm Tr} \left( \Re D_k\right)^2\sim k$. This finishes the proof of \eqref{eq:D_HS}.
\end{proof}

\begin{proof}[Proof of Lemma \ref{lem:Im}]
Firstly notice that $\lVert (H-\ii\eta(I+S))^{-1}\rVert\le 1/\eta$. Indeed, we have
\begin{equation}
	\begin{split}
		&\lVert (H-\ii\eta(I+S))^{-1}\rVert = \left\lVert (I+S)^{-1/2}\left((I+S)^{-1/2}H(I+S)^{-1/2}-\ii\eta\right)^{-1}(I+S)^{-1/2}\right\rVert\\
		&\quad \le \left\lVert \left((I+S)^{-1/2}H(I+S)^{-1/2}-\ii\eta\right)^{-1}\right\rVert\le 1/\eta.
	\end{split}
	\label{eq:G_S}
\end{equation}
Next, in order to verify \eqref{eq:Im} we need to show that
\begin{equation*}
	\langle \bfv, \Im (H-\ii\eta)^{-1}\bfv\rangle\sim \langle \bfv, \Im (H-\ii\eta(I+S))^{-1}\bfv\rangle  
\end{equation*}
uniformly in $\bfv\in\C^N$. This is equivalent to
\begin{equation}
	\langle(H-\ii\eta)^{-1} \bfv, (H-\ii\eta)^{-1} \bfv\rangle \sim \langle(H-\ii\eta(I+S))^{-1} \bfv, (I+S)(H-\ii\eta(I+S))^{-1} \bfv\rangle.
	\label{eq:Im_vect} 
\end{equation}
Denote $\bfw:= (H-\ii\eta(I+S))^{-1} \bfv$. Then \eqref{eq:Im_vect} reads as
\begin{equation}
	\lVert \bfw-\ii\eta(H-\ii\eta)^{-1}S\bfw\rVert^2 \sim \langle \bfw, (I+S)\bfw\rangle
	\label{eq:Im_vect1}
\end{equation}
The left-hand-side of \eqref{eq:Im_vect1} has an upper bound of order $\lVert \bfw\rVert^2$, since $\lVert (H-\ii\eta)^{-1}\rVert\le 1/\eta$. On the other hand, the right-hand-side of \eqref{eq:Im_vect1} is of order $\lVert \bfw\rVert^2$. Therefore, we have shown that
\begin{equation*}
	\langle(H-\ii\eta)^{-1} \bfv, (H-\ii\eta)^{-1} \bfv\rangle \lesssim \langle(H-\ii\eta(I+S))^{-1} \bfv, (I+S)(H-\ii\eta(I+S))^{-1} \bfv\rangle.
\end{equation*}
The proof of \eqref{eq:Im_vect} in the opposite direction goes exactly along the same lines, we only need to take $\bfw:= (H-\ii\eta)^{-1} \bfv$ and to use \eqref{eq:G_S}. This finishes the proof of Lemma \ref{lem:Im}. 
\end{proof}

\section{Weakly correlated Dyson Brownian motions at the edge: Proof of Proposition \ref{prop:DBMindep}}
\label{sec:wcDBM}

In this section, we prove Proposition \ref{prop:DBMindep} by constructing (truly) \emph{independent} comparison processes ${\bm \mu}^{(l)}(t)$, $l \in [2]$, and we will couple them to the eigenvalues of $H^{(N)}, H^{(N-k)}$, closely following\footnote{We skip several technical details such as regularization of the driving process in \eqref{eq:init_DBM} and discuss only essential steps of construction. For more details see \cite{macroCLT_complex, macroCLT_real}.} \cite[Sec.~7]{macroCLT_complex} and \cite[Sec.~7]{macroCLT_real}. 
First, we introduce a small $N$-independent cut-off parameter $0<\omega_A\le 1$ which will be chosen later in Proposition \ref{prop:DBM}. In order to simplify the notations we assume that $N^{\omega_A}\in\N$. Consider the $2N^{\omega_A}$-dimensional process
\begin{equation}
\underline{{\bm b}}(t):=\left(b^{(1)}_1(t),\ldots, b^{(1)}_{N^{\omega_A}}(t),b^{(2)}_1(t),\ldots, b^{(2)}_{N^{\omega_A}}(t)\right)
\end{equation} 
with $b_i^{(l)}$ from \eqref{eq:BM_corr}, which is \emph{not} a standard $2N^{\omega_A}$-dimensional Brownian motion due of the 
non-trivial correlations between the first half and the second half of the vector $\underline{{\bm b}}(t)$; cf. \eqref{eq:BM_corr}. 
However, $\underline{{\bm b}}(t)$ is still a martingale. Moreover, by Doob's martingale representation theorem from \cite[Theorem 18.12]{Kallenberg} there exists a standard $2N^{\omega_A}$-dimensional Brownian motion $\underline{{\bm \beta}}(t)$ such that
\begin{equation}
\underline{{\bm b}}(t) = \sqrt{C(t)}\underline{\bm \beta}(t),\quad \underline{\bm \beta}(t)=\left(\beta_1^{(1)}(t),\ldots, \beta^{(1)}_{N^{\omega_A}}(t),\beta^{(2)}_1(t),\ldots, \beta^{(2)}_{N^{\omega_A}}(t)\right),
\end{equation}
where $C(t)$ is an $N^{2\omega_A}\times N^{2\omega_A}$ non-negative matrix defined by
\begin{equation*}
C_{ij}(t) := \Theta^{l_1l_2}_{\mathfrak{i}\mathfrak{j}}(t), \quad i=(l_1-1)N^{\omega_A}+\mathfrak{i},\, j=(l_2-1)N^{\omega_A}+\mathfrak{j} 
\end{equation*}
and $\Theta_{\mathfrak{i} \mathfrak{j}}^{l_1 l_2}$ was introduced in \eqref{eq:BM_corr}.

Next, consider independent standard Brownian motions $\beta^{(1)}_i$ and $\beta^{(2)}_j$ for $i\in [N^{\omega_A}+1,N_1]$ and $j\in [N^{\omega_A}+~1, N_2]$. These Brownian motions are also independent of the collection $\lbrace b^{(l)}_j\rbrace_{j=1}^{N_l}$ for $l=1,2$. Let $W_G^{(1)}$ and $W_G^{(2)}$ be independent GOE/GUE matrices of sizes $N_1\times N_1$ and $N_2\times N_2$ respectively, which are further assumed to be independent from all previously defined processes. 

Finally, define the comparison process ${\bm\mu}^{(l)}_t$ as the strong solution to the DBM
\begin{equation}
\dif \mu^{(l)}_i(t) = \frac{\dif \beta^{(l)}_i(t)}{\sqrt{\beta N}} + \frac{1}{N}\sum_{j\neq i} \frac{1}{\mu^{(l)}_j(t) -\mu^{(l)}_i(t)}\dif t,\quad  i\in [N_l],
\label{eq:comp_proc}
\end{equation}
with initial data ${\bm\mu}^{(l)}(0)$ being the eigenvalues of $\sqrt{N_l/N}\,W^{(l)}_G$. Now we are ready to state the main technical result of this section, from which we immediately deduce Proposition \ref{prop:DBMindep}. 

\begin{proposition}[Comparison with independent eigenvalues]\label{prop:DBM} Fix $0<\omega_A<\omega_0/2<1$ and $\omega_{{\rm ov}}>0$. Denote $t_1:=N^{-1/3+\omega_0}$ and assume that
\begin{equation}
\sup_{t\in[0,t_1]}C_{ij}(t)\prec N^{-\omega_{{\rm ov}}}, \quad i,j\le N^{\omega_A}.
\label{eq:ass_overlap}
\end{equation}
Then there exist $\omega, \widehat{\omega}>0$ with $\widehat{\omega}\le \omega/10\le \omega_0/100$, such that, with very high probability, 
\begin{equation}
\vert \lambda^{(l)}_i(t_1)-\mu^{(l)}_i(t_1)\vert \le N^{-2/3-\omega},\quad i\le N^{\widehat{\omega}},\quad l\in[2]\,. 
\label{eq:coulpl_aim}
\end{equation}
\end{proposition}

\begin{proof}[Proof of Proposition~\ref{prop:DBMindep}]
The approximate independence in \eqref{eq:step2} immediately follows from \eqref{eq:coulpl_aim}, the independence of the $\mu_i^{(l)}$'s (by construction of the processes in \eqref{eq:comp_proc}), and the smoothness of $F,G$, which allows us to perturb their arguments by an amount of order $N^{-\omega}$ without essentially changing their value.
\end{proof}

It now remains to give the proof of Proposition \ref{prop:DBM}. 

\begin{proof}[Proof of Proposition \ref{prop:DBM}]
We independently verify \eqref{eq:coulpl_aim} for $l=1$ and $l=2$ without relying on the connection between \eqref{eq:BM_corr} for these two values of parameter $l$. 
Since the arguments for $l=1$ and $l=2$ are identical, we further focus on the case $l=1$ in our analysis of \eqref{eq:BM_corr} and \eqref{eq:comp_proc}. Henceforth, we omit the superscript in $\lambda^{(1)}_j(t)$ and $\mu^{(1)}_j(t)$ for simplicity. Our proof closely follows \cite{Bou_extreme} (see also \cite{landon2017edge}), so in order to avoid redundancy we will show how the main steps in \cite{Bou_extreme} are modified to suit the current setting.

First of all, we interpolate between the weakly correlated DBM \eqref{eq:BM_corr} and the comparison process \eqref{eq:comp_proc} with independent driving Brownian motions by introducing the following one-parameter family of flows: for $\alpha\in [0,1]$ we set (recall $\beta=1$ for real symmetric case, $\beta=2$ for complex Hermitian case)
\begin{equation}
	\dif z_i(t,\alpha) = \alpha \sqrt{\frac{\beta}{2N}} \nc \dif b_i(t) + (1-\alpha)  \sqrt{\frac{\beta}{2N}} \nc \dif\beta_i(t) + \frac{1}{N}\sum_{j\neq i}  \frac{1}{z_j(t,\alpha) - z_i(t,\alpha)}, \quad i\in [N],
	\label{eq:z_flow}
\end{equation}
with initial conditions $z_i(0,\alpha):=\alpha \lambda_i(0) + (1-\alpha)\mu_i(0)$. In particular, $z_i(t,0)=\mu_i(t)$ and $z_j(t,1)=\lambda_j(t)$ for all $t\ge 0$. The process \eqref{eq:z_flow} is well-posed in the sense  particle collisions almost surely do not occur, i.e.
\begin{equation}
	\mathbf{P}\left[z_1(t,\alpha)< z_2(t,\alpha)<\cdots< z_N(t,\alpha), \quad \forall t\in [0,1], \,\forall \alpha\in [0,1]\right]=1. 
	\label{eq:z_well-posed}
\end{equation}
The proof of \eqref{eq:z_well-posed} is analogous to the proof of \cite[Theorem 5.2]{local_stat_addition} adapted to the current context as outlined in \cite[Section 7.6]{macroCLT_real}. We point out that in the real case ($\beta=1$)  to ensure the well-posedness of the interpolation process we need to slightly change the flow in \eqref{eq:z_flow} by dividing the stochastic terms by factor $1+N^{-\omega}$, for a sufficiently small $\omega>0$ (see e.g. \cite[Eq. (4.13)]{local_stat_addition}, \cite[Eq. (7.22)]{macroCLT_real}).\nc

Another important property of \eqref{eq:z_flow} essential for the proof of \eqref{eq:coulpl_aim} is the rigidity. To explain this phenomenon let us denote the Stieltjes transform of the empirical particle density by 
\begin{equation*}
	m_N(z,t,\alpha):=\frac{1}{N}\sum_{j=1}^N \frac{1}{z_j(t,\alpha)-z}.
\end{equation*}
Additionally, let $m_{sc}$ stand for the Stieltjes transform of the semicircular distribution $\rho_{sc}(x)=\frac{1}{2\pi}\sqrt{[4-x^2]_+}$. Recall the definition of $\gamma_j^{(N-k)}$ from \eqref{eq:quantiles} and set $\gamma_j:=\gamma_j^{(N)}$. 

Then for any (large) $L>0$ and (small) $\varepsilon>0$ we have
\begin{align}
	\sup_{\vert\Re z\vert\le L}\sup_{N^{-1+\varepsilon}\le\rho_{sc}(z)\Im z\le L} \sup_{\alpha\in [0,1]}\sup_{t\in[0,t_1]} N\Im z\vert m_N(z,t,\alpha)-m_{sc}(z)\vert \prec 1, \label{eq:m_rigidity}\\
	\sup_{j\le N^{\omega_A}}\sup_{\alpha\in [0,1]}\sup_{t\in [0,t_1]} N^{2/3}\vert z_j(t,\alpha)-\gamma_j\vert\prec 1\label{eq:g_rigidity}.
\end{align}
The rigidity estimate \eqref{eq:g_rigidity} follows as a standard consequence of \eqref{eq:m_rigidity} (see e.g. \cite[Section 5]{EYY12}). The proof of \eqref{eq:m_rigidity} is identical to that of \cite[Theorem 3.1]{HL_rig}. The only modification required is in the proof of \cite[Proposition 3.8]{HL_rig}, as the driving Brownian motions in \eqref{eq:z_flow} are weakly correlated. This correlation introduces cross-terms between different particles, contributing to the quadratic variation in \cite[Eq. (3.33)]{HL_rig} and to the sum in the left-hand-side of \cite[Eq. (3.32)]{HL_rig}. However, this additional terms are negligible, which is ensured by \eqref{eq:ass_overlap} combined with the choice $\omega_A<\omega_{\mathrm{ov}}/2$ (see \cite[Section 4.4]{bourgade2024fluctuations} for a similar argument in the bulk regime).

Having \eqref{eq:z_well-posed}--\eqref{eq:g_rigidity} in hand, the rest of the proof closely follows\footnote{We point out that the DBM in \cite[Eq. (1.8)]{Bou_extreme} compared to \eqref{eq:init_DBM} contains an additional term $-\lambda_i^{(l)}(t)/2$. However, inspecting the proof in \cite[Section 2.2]{Bou_extreme}, it is clear that removing this term does not change anything in the analysis.} \cite[Sections 2.2--2.3]{Bou_extreme}.  The only complication compared to \cite{Bou_extreme} is that an additional term containing the quadratic variation process $\dif b_i-\dif\beta_i$ appears in the right-hand-side of \cite[Eq. (2.3)]{landon2017edge} (see also \cite[Lemma 4.19]{bourgade2024fluctuations} for an analogous term in a slightly different model)\nc. However, a similar estimate has been already conducted for $\dif \xi_{1,i}$ in \cite[Eqs. (7.92)--(7.95)]{macroCLT_complex}  and in the proof of \cite[Lemma 4.20]{bourgade2024fluctuations}\nc. This completes the proof of Proposition \ref{prop:DBM}.
\end{proof}

\section{Green function comparison: Proof of Propositions \ref{prop:GFTeigenval} and \ref{prop:GFTeigenval2}}
\label{sec:GFT}

In this section, we provide the proofs of Propositions \ref{prop:GFTeigenval} and \ref{prop:GFTeigenval2} based on a Green function comparison (GFT) argument. At the spectral edge, it is a well established approach (see, e.g., \cite[Section 17]{erdHos2017dynamical} or \cite{EYY12}) to directly compare the matrix of interest with the corresponding GOE/GUE ensemble (a so called \emph{long time GFT}). We follow this approach here as well, the only novelty being that now the resolvents of \emph{two different} Wigner matrices are involved in the GFT (see Proposition \ref{prop:GFT} below).

To prepare the formulation of the GFT, let $H^{(N)}$ be a Wigner matrix and consider the matrix valued Ornstein-Uhlenbeck process
\begin{equation} \label{eq:OU}
\dif \widehat{H}^{(N)}_t =  - \frac{\widehat{H}^{(N)}_t}{2} \dif t+\frac{\dif \widehat{B}^{(N)}_t}{\sqrt{N}} \,, \qquad \widehat{H}_0^{(N)} = H^{(N)} \,, 
\end{equation}
where $\widehat{B}^{(N)}_t$ is a real symmetric or complex Hermitian standard matrix valued Brownian motion of the same symmetry class as $H^{(N)}$ (see, e.g., below \eqref{eq:H_flow}) and we observe that $\widehat{H}^{(N)}_\infty = \mathrm{G}\beta \mathrm{E}$. Moreover, similarly to \eqref{eq:H_min_flow}, we consider the time-dependent analog of the minor process defined in \eqref{eq:minordef} and denote the resolvent of $\widehat{H}^{(N-k)}_t$ for $k \in [N]$ at a spectral parameter $z \in \C \setminus \R$ by $G_t^{(N-k)}(z) := \big(\widehat{H}_t^{(N-k)} - z\big)^{-1}$.

\begin{proposition}[Green function comparison at the edge] \label{prop:GFT}
Fix $\delta > 0$. 	Take $N \in \N$ and $k \in [N]$ with $k \le N^{1 - \delta}$. 
Consider $f,\widetilde{f} : \R\to \R$ satisfying 
\begin{equation}
\max_{x \in \R} |f^{(\alpha)}(x)| (|x| +1)^{-C_1}  + \max_{x \in \R} |\widetilde{f}^{(\alpha)}(x)| (|x| +1)^{-C_1}\le C_1 \,, \ i \in [2]\,, \ \alpha = 1,2,3 \,, 
\end{equation}
for some constant $C_1 > 0$. Then there exists $\varepsilon_0 > 0$, depending only on $C_1$, such that we have the following: For any $\varepsilon \in (0,\varepsilon_0)$ and $E_i, \widetilde{E}_i \in \R$, $i \in [2]$ satisfying 
\begin{equation*}
|E_i - 2| \le C N^{-2/3 + \varepsilon} \quad \text{and} \quad |\widetilde{E}_i - 2\sqrt{1 - k/N}| \le C N^{-2/3 + \varepsilon} \,, \quad i \in [2] \,, 
\end{equation*}
and setting $\eta := N^{-2/3 - \varepsilon}$, we have 
\begin{equation}
	\begin{split}
	\label{eq:gftneed}
\bigg| &\E f\bigg( N \int_{E_1}^{E_2} \dif x \, \langle \Im G_t^{(N)}(x +  \ii \eta)\rangle\bigg) \widetilde{f}\bigg( N \int_{\widetilde{E}_1}^{\widetilde{E}_2} \dif x \, \langle \Im G_t^{(N-k)}(x +  \ii \eta)\rangle\bigg)  \\
&- \E f\bigg( N \int_{E_1}^{E_2} \dif x \, \langle \Im G_s^{(N)}(x +  \ii \eta)\rangle\bigg) \widetilde{f}\bigg( N \int_{\widetilde{E}_1}^{\widetilde{E}_2} \dif x \, \langle \Im G_s^{(N-k)}(x +  \ii \eta)\rangle\bigg) \bigg| \le C N^{-1/6 + C \varepsilon}
	\end{split}
\end{equation}
for all $t,s \in [0, \infty]$, for some constant $C$ and large enough $N$. 
\end{proposition}

Armed with Proposition \ref{prop:GFT}, we can immediately prove Propositions \ref{prop:GFTeigenval} and \ref{prop:GFTeigenval2}.  In the remainder of this section we will use Proposition~\ref{prop:GFT} to identify the distribution of individual eigenvalues close to the edge of the spectrum. Translating \eqref{eq:gftneed} into a comparison of 
statistics of eigenvalues is a standard argument 
(see e.g. \cite[Section 15.2]{erdHos2017dynamical}), and so omitted. \nc

\begin{proof}[Proof of Proposition~\ref{prop:GFTeigenval}]
Consider the Ornstein-Uhlenbeck flow \eqref{eq:OU}
along which the first and second moments of $\widehat{H}_t^{(N)}$ are preserved. Furthermore, we have
\begin{equation}
\label{eq:oudistr}
\widehat{H}_t^{(N)}\stackrel{\dif}{=}e^{-t/2}\widehat{H}_0^{(N)}+\sqrt{1-e^{-t}}U,
\end{equation}
for a GUE/GOE matrix $U$ independent of $\widehat{H}_0^{(N)}$. Next, we consider another flow (note that there is no drift term, unlike in \eqref{eq:OU})
\begin{equation}
\label{eq:matDBM}
\dif H_t^{(N)}=\frac{\dif B_t^{(N)}}{\sqrt{N}}, \qquad\quad H_0^{(N)}=e^{-t_1/2}H^{(N)},
\end{equation}
where $t_1=N^{-1/3+\omega_1}$. It is then easy to see that
\begin{equation}
H_t^{(N)}\stackrel{\dif}{=}H_0^{(N)}+\sqrt{t}\widetilde{U},
\end{equation}
for a GUE/GOE matrix $\widetilde{U}$ independent of $H_0^{(N)}$.

The flows \eqref{eq:OU} and \eqref{eq:matDBM} induce a flow on the minors $H_t^{(N-k)}$, $\widehat{H}_t^{(N-k)}$ (see \eqref{eq:H_flow}--\eqref{eq:H_min_flow}). Using the shorthand notation $N_1:=N$, $N_2:=N-k$, we thus have
\begin{equation}
\label{eq:reldistr}
H_{ct_1}^{(N_l)}\stackrel{\dif}{=} \widehat{H}_{t_1}^{(N_l)}, \qquad\quad l=1,2,
\end{equation}
for $c=c(t_1):=(1-e^{-t_1})/t_1$.

We are now ready to conclude the proof. Denote the eigenvalues of $\widehat{H}_t^{(N_l)}$ by $\widehat{\lambda}_i^{(l)}(t)$, for $l=1,2$, and define the short--hand notations
\begin{equation}
\Lambda_i^{(l)}(t):=(\lambda_i^{(l)}(t)-2(N_l/N_1)^{1/2})N_l^{2/3}, \quad \widehat{\Lambda}_i^{(l)}(t):=(\widehat{\lambda}_i^{(l)}(t)-2(N_l/N_1)^{1/2})N_l^{2/3}, \quad l=1,2,
\end{equation}
with $\lambda_i^{(l)}(t)$ being the eigenvalues of $H_t^{(N_l)}$ 
\nc

Then, using \eqref{eq:gftneed} and proceeding as in the proof of \cite[Theorem 15.2]{erdHos2017dynamical}, we find
\begin{equation}
\begin{split}
\label{eq:gftindepneed}
\E\left[F\big(\Lambda_i^{(1)}(0)\big)G\big(\Lambda_j^{(2)}(0)\big)\right]=\E\left[F\big(\widehat{\Lambda}_i^{(1)}(t_1)\big)G\big(\widehat{\Lambda}_j^{(2)}(t_1)\big)\right]+\mathcal{O}\big(N^{-\delta}\big),
\end{split}
\end{equation}
for some small fixed $\delta>0$. Then, by \eqref{eq:reldistr}, we have
\begin{equation}
\begin{split}
\label{eq:almtheind}
\E\left[F\big(\widehat{\Lambda}_i^{(1)}(t_1)\big)G\big(\widehat{\Lambda}_j^{(2)}(t_1)\big)\right]=\E\left[F\big({\Lambda}_i^{(1)}(ct_1)\big)G\big({\Lambda}_j^{(2)}(ct_1)\big)\right]+\mathcal{O}\big(N^{-\delta}\big).
\end{split}
\end{equation}
Using that by assumption the independence \eqref{eq:step2} holds for the right-hand-side of \eqref{eq:almtheind}, and using \eqref{eq:reldistr} again, we thus conclude
\begin{equation}
\begin{split}
\E\left[F\big(\widehat{\Lambda}_i^{(1)}(t_1)\big)G\big(\widehat{\Lambda}_j^{(2)}(t_1)\big)\right]=\E\left[F\big(\widehat{\Lambda}_i^{(1)}(t_1)\big)\right]\E\left[G\big(\widehat{\Lambda}_j^{(2)}(t_1)\big)\right]+\mathcal{O}\big(N^{-\delta}\big).
\end{split}
\end{equation}
This, together with another similar application of \eqref{eq:gftindepneed} (this time for the product of expectations), concludes the proof.
\end{proof}

\begin{proof}[Proof of Proposition \ref{prop:GFTeigenval2}]
Given Proposition \ref{prop:GFT}, the proof of Proposition \ref{prop:GFTeigenval2} follows standard arguments; see, e.g., \cite[Section 15.2]{erdHos2017dynamical}. 
\end{proof}
It thus remains to give the proof of Proposition \ref{prop:GFT}. 
\begin{proof}[Proof of Proposition \ref{prop:GFT}]
	First, we split the time interval $[0, \infty]$ in two parts, $I_1 := [0,T]$ and $I_2 := [T, \infty]$, for $T := 100 \log N$. Then we distinguish four cases: For $t,s \in I_2$, the claim is immediate by trivial perturbation theory, i.e.~compare $s,t \ge T$ directly with GOE/GUE. In fact, in this regime by \eqref{eq:oudistr} it follows that the distribution $\widehat{H}_t$ is entry-wise close to a GUE/GOE matrix: $\lVert [e^{-t/2}\widehat{H}_0+\sqrt{1-e^{-t}}U]-U\rVert\lesssim N^{-10}$. This bound, together with Weyl's inequality gives the desired result. The two cases $t \in I_1$ and $s \in I_2$ (and vice versa) can be reduced to the fourth case $t,s \in I_1$ via the same trivial perturbation argument. 
	
	We will now discuss the fourth, and last, case. For simplicity, we will henceforth assume $f(x) = \widetilde{f}(x) = x$ and prove Proposition \ref{prop:GFT} without integration. The adaptions to the more general situation are routine (see, e.g., \cite[Proof of Theorem 17.4]{erdHos2017dynamical}) and so omitted. More precisely, for $E, \tilde{E} \in \R$ satisfying $|E - 2| \le N^{-2/3 + \varepsilon}$ and $|\widetilde{E} - 2\sqrt{1 - k/N}| \le N^{-2/3 + \varepsilon}$, we will show that, for $z = E + \ii \eta$ and $\tilde{z} = \widetilde{E} + \ii \eta$, it holds that
\begin{equation} \label{eq:goalderivative}
(N \eta)^2 \bigg| \frac{\dif}{\dif t}\E  \langle \Im G_t^{(N)}(z)\rangle \,  \langle \Im G_t^{(N-k)}(\tilde{z})\rangle \bigg| \lesssim N^{-1/6 + C \varepsilon} \,, \quad t \in [0,T] \,. 
\end{equation}

To do so, we will use that, for a smooth function $f$ and $H_t$ solving \eqref{eq:OU}, by Itô's formula, it holds that
	\begin{equation} \label{eq:Ito}
		\frac{\dif}{\dif t} \E f(H_t) = - \frac{1}{2} \sum_{a,b} \E h_{ab}(r) (\partial_{ab} f)(H_t) + \frac{1}{2} \sum_{a,b} \sum_{\alpha \in \{ab, ba\}}\kappa(ab,\alpha) \E (\partial_{ab}\partial_{\alpha} f)(H_t) \,, 
	\end{equation}
	and hence by a cumulant expansion (see, e.g., \cite[Prop.~3.2]{slow_corr})
	\begin{equation} \label{eq:cumexp}
		\frac{\dif}{\dif t} \E f(H_t) = \sum_{m=2}^{K-1} \sum_{a,b} \sum_{\bm \alpha \in \{ab,ba\}^m} \frac{\kappa(ab, \bm \alpha)}{m!} \E (\partial_{ab} \partial_{\bm \alpha} f)(H_t) + \Omega_K
	\end{equation}
	with some explicit error term $\Omega_K$. Here, for a $m$-tuple of double indices $\bm \alpha = (\alpha_1, ... , \alpha_m)$ we used the shorthand notation $\kappa(ab, (\alpha_1, ... , \alpha_m)) = \kappa(h_{ab}, h_{\alpha_1}, ... , h_{\alpha_m})$ for the joint cumulant of $h_{ab}, h_{\alpha_1}, ... , h_{\alpha_m}$ and set $\partial_{\bm \alpha} = \partial_{h_{\alpha_1}} ... \partial_{h_{\alpha_m}}$ and $\partial_{ab} = \partial_{h_{ab}}$. 
	
	In order to simplify the following presentation, we will henceforth assume that there is no difference for off-diagonal ($a \neq b$) and diagonal ($a = b$) cumulants $\kappa(ab, \bm \alpha)$ in \eqref{eq:cumexp}. The general case can be handled with straightforward minor modifications and is thus left to the reader. Moreover, since $t$ is fixed throughout the following argument, we will drop the subscript. For further ease of notation, we will write $G \equiv G_t^{(N)}(z)$ and $\widetilde{G} \equiv G_t^{(N-k)}(\tilde{z})$. 

By Itô's Lemma \eqref{eq:Ito} for $f(H) = \langle \Im G \rangle \langle \Im \widetilde{G} \rangle$, cumulant expansion \eqref{eq:cumexp} (truncated at at some large but finite order $R \in \N$), and using that $|\kappa(ab, (\alpha_1,..., \alpha_m))| \lesssim N^{-(m+1)/2}$ in \eqref{eq:cumex}, we find 
\begin{equation} \label{eq:cumex}
 \bigg| \frac{\dif}{\dif t}\E  \langle \Im G\rangle \,  \langle \Im \widetilde{G}\rangle \bigg| \lesssim \sum_{m=3}^R N^{-m/2} \ \sum_{l=0}^m \bigg| \sum_{a,b}  \E \partial_{ab}^l \partial_{ba}^{m-l} \langle \Im G\rangle \,  \langle \Im \widetilde{G}\rangle  \bigg| + \mathcal{O}(N^{-100}) \,. 
\end{equation}

Only the terms of order $m=3$ in \eqref{eq:cumex} have to be estimated carefully; the terms of order $m \ge 4$ will be treated starting from \eqref{eq:thirdex}. Now, for $m=3$, we consider one exemplary term, where $\partial_{ab}^l \partial_{ba}^{m-l}$ hit only one of the $\langle \Im G \rangle$ factors and leaves the other one intact, and another exemplary term, where $\partial_{ab}^l \partial_{ba}^{m-l}$ acts on both $\langle \Im G \rangle$ factors. As it is irrelevant for our estimates, we will henceforth ignore the difference between $G$ and $\widetilde{G}$, as well as their adjoints $G^*$ and $\widetilde{G}^*$, respectively, which arise by differentiation of $\Im G$. This is because, for the differentiated terms, we will neither exploit so-called \emph{Ward summations} for different matrices $G$ and $\widetilde{G}$, nor the smallness of $\Im G$ compared to $G$.  Moreover, the different normalization of the traces, $N^{-1}$ and $(N-k)^{-1}$, respectively, are irrelevant for our estimates. 

Now, as the first exemplary term with $m=3$ and $l = 1$, we estimate (an additional $N^{-1}$ factor stems from the normalized trace)
\begin{equation} \label{eq:firstex}
	\begin{split}
&N^{-5/2} \bigg| \E \sum_{j, a,b}  G_{ja} G_{bb} G_{aa} G_{bj} \langle \Im G \rangle \bigg|  \\
\lesssim  \, &N^{-5/2} \bigg| \sum_{j,a,b} \E G_{ja} G_{bj} \langle \Im G \rangle \bigg| + N^{-5/2} N^{-1/3+ 2\varepsilon} \E \sum_{j,a,b}  |G_{ja} G_{bj}| \langle \Im G \rangle \\
\lesssim \, & N^{-5/2} N^{-1/3+ 2\varepsilon} \big| \E \big( G^2\big)_{\bm 1 \bm 1} \big| +  N^{-5/2} N^{-2/3+ 4\varepsilon} \E \sum_{a,b} \sqrt{(GG^*)_{aa}} \sqrt{(GG^*)_{bb}} \\
\lesssim \, & N^{-1/2} N^{-1/3+ 5\varepsilon}/(N \eta)^2 + N^{-1/6+ 7 \varepsilon} /(N \eta)^2 \lesssim N^{-1/6 + 7 \varepsilon}/(N \eta)^2
	\end{split}
\end{equation}
where $\bm 1 = (1,...,1) \in \C^N$. 
In the first step, we used the usual isotropic local law for Wigner matrices $G_{bb} = m_{\rm sc} + \mathcal{O}_{\prec}(N^{-1/3+\varepsilon})$, where $m_{\rm sc} = m_{\rm sc}(z)$ denotes the Stieltjes transform of the semicircular density (see, e.g., \cite[Theorem~6.7]{erdHos2017dynamical}). To go to the third line, we employed \emph{isotropic resummation} in the form
$$
\sum_{ab} (G^2)_{ab} = (G^2)_{\bm 1 \bm 1} \,, 
$$
the averaged local law in the form $\langle \Im G \rangle = \mathcal{O}_{\prec}(N^{-1/3+ \varepsilon})$ (see, e.g., \cite[Theorem~6.7]{erdHos2017dynamical}), where we additionally used that $|\Im m_{\rm sc}| \lesssim N^{-1/3+\varepsilon}$ at the edge regime, and a Schwarz inequality. In the penultimate step, we again used a Schwarz inequality, a Ward identity $GG^* = \Im G/\eta$, an isotropic law $(\Im G)_{bb} = \mathcal{O}_{\prec}(N^{-1/3+\varepsilon})$, and $\eta = N^{-2/3 - \varepsilon}$. 

Next, as the second exemplary term with $m=3$ and $l =2$, we estimate (two additional $N^{-1}$ factors stem from two normalized traces)
\begin{equation} \label{eq:secondex}
\begin{split}
&N^{-7/2} \bigg| \E \sum_{i,j, a,b}  G_{ia} G_{bb}  G_{ai}  G_{jb} G_{aj} \bigg|  \\
\lesssim \, & N^{-7/2} \E \sum_{i,j,a,b} |G_{ia} G_{ai} G_{jb} G_{aj}| \lesssim  N^{-1/6 + 2 \varepsilon} /(N\eta)^2
\end{split}
\end{equation}
similarly to \eqref{eq:firstex}, my means of Schwarz inequalities (for the $i$ and $j$ summations; the $a$ and $b$ summations are estimated trivially by $N^2$), Ward identities and the single resolvent laws from above. 

Now we consider terms in \eqref{eq:cumex} of order $m \ge 4$. If only one $\langle \Im G \rangle$ factor is hit by a derivative (cf.~\eqref{eq:firstex}), we estimate
\begin{equation} \label{eq:thirdex}
	\begin{split}
&N^{-(m+2)/2} \E \sum_{j,a,b} |G_{ja}...G_{bj} \langle \Im G \rangle| 
\lesssim N^{-(m+2)/2} \E \sum_{j,a,b} |G_{ja}G_{bj}  \langle \Im G \rangle| \\
 \lesssim \, &N^{-(m-4)/2} N^{-1/3 + 5 \varepsilon}/(N \eta)^2 \le N^{-1/3 + 5 \varepsilon}/(N \eta)^2
	\end{split}
\end{equation}
again by single resolvent laws and a Schwarz inequality together with a Ward identity. If both $\langle \Im G \rangle$ factors are hit by a derivative (cf.~\eqref{eq:secondex}), we estimate, similarly to \eqref{eq:thirdex}, 
\begin{equation*}
\begin{split}
N^{-(m+4)/2} \E \sum_{i,j,a,b} |G_{ia}...G_{bi} G_{ja} ... G_{bj}| \lesssim \,  &N^{-(m+4)/2} \E \sum_{i,j,a,b} |G_{ia}G_{bi} G_{ja}  G_{bj}| \\
\lesssim  \, & N^{-(m-4)/2} N^{-2/3 + 2 \varepsilon}/(N \eta)^2 \lesssim N^{-2/3 + 2 \varepsilon}/(N \eta)^2 \,. 
\end{split}
\end{equation*}
We have thus shown \eqref{eq:goalderivative}, which concludes the proof of Proposition \ref{prop:GFT}. 
\end{proof}

\section{Subcritical regime: Proof of Theorem \ref{theo:small_k}}
\label{sec:subcrit}

In this section, we provide the proof of Theorem \ref{theo:small_k}. We begin with the eigenvalue fluctuations. 

\begin{proposition}[Subcritical regime: Normal fluctuation on scale $\sqrt{k}/N$] \label{EigenvalueSub}Fix a (small) constant $\varepsilon>0$. Let $H$ be a Wigner matrix satisfying Assumption \ref{ass:momass} and let $k\in\N$ be such that $ 1\ll  k\le N^{2/3-\varepsilon}$. Then for any fixed $i$, $\lambda_i^{(N)} - \lambda_i^{(N-k)}$ is asymptotically normal (see Footnote \ref{ftn:asymptnorm}) with normalization $N/\sqrt{k}$, i.e.
$$
  \frac{N(\lambda_i^{(N)} - \lambda_i^{(N-k)}) - k}{\sqrt{k}}  \sim \mathcal{N}(0,2/\beta)
$$
where $\beta=1$ (real case) and $\beta=2$ (complex case) and $\mathcal{N}(0,2/\beta)$ is a real centered Gaussian random variable of variance $2/\beta$. 
\nc\end{proposition}
The proof of Proposition \ref{EigenvalueSub} is given in Section \ref{subsec:evalsub}. The following Proposition \ref{EigenvectorSub} concerns the eigenvector overlaps. Its proof is given in Section \ref{subsec:evecsub}. 
\begin{proposition}[Subcritical regime: Eigenvector overlap] \label{EigenvectorSub} Fix a (small) constant $\varepsilon>0$. Let $H$ be a Wigner matrix satisfying Assumption \ref{ass:momass} and let $k\in\N$ be such that $  k\le N^{2/3-\varepsilon}$. 
Let $\bm{w}_i^{(N-k)}$ denote the $i$-th normalized eigenvector of its $(N-k)$-minor.
Then for any fixed $i$, any small (but fixed) constant $\delta>0$, there exists a positive constant $a>0$, such that, for $N$ large enough,
\begin{align*}
\mathbf{P}\Big(0\leq 1-|\langle \bm{w}_i^{(N)}, {\bm{w}}_i^{(N-k)}\rangle|\leq k N^{-\frac23+\delta}\Big)\geq 1-N^{-a}\,. 
\end{align*}
\end{proposition}
Propositions \ref{EigenvalueSub} and \ref{EigenvectorSub} for $i=1$ immediately prove parts (a) and (b) of Theorem \ref{theo:small_k}, respectively. \qed

\subsection{Proof of Proposition \ref{EigenvalueSub}} \label{subsec:evalsub}
The proof of Proposition \ref{EigenvalueSub} is split in several parts. 

\subsubsection{Preliminaries}
For any $n\geq 2$, we write 
\begin{align*}
	H^{(n)}=\left( \begin{array}{cc}
		h^{(n)}_{11}  & (\bm{a}^{(n)})^*\\ \\
		\bm{a}^{(n)}   & H^{(n-1)}
	\end{array}\right)
\end{align*}
Note that here we rewrite the $N-n+1$-th diagonal entry of $H$ as $h_{11}^{(n)}$.  Further, we denote the 
normalized  eigenvector associated with $\lambda_i^{(n)}$ by $\bm{w}^{(n)}_i$ and we write
\begin{align}
	\bm{w}^{(n)}_i=\left(
	\begin{array}{c}
		u^{(n)}_i \\   \\
		\bm{v}^{(n)}_i
	\end{array}
	\right), \label{19071010}
\end{align}
where $u^{(n)}_i$ is the first component of $\bm{w}^{(n)}_i$.  From the eigenvalue equation $H^{(n)} \bm{w}^{(n)}_i= \lambda^{(n)}_i \bm{w}^{(n)}_i$, 
we get 
\begin{equation} \label{19070901}
	\begin{split}
	&h^{(n)}_{11} u^{(n)}_i+ (\bm{a}^{(n)})^* \bm{v}^{(n)}_i= \lambda^{(n)}_i u^{(n)}_i,\\
& u^{(n)}_i \bm{a}^{(n)}+ H^{(n-1)} \bm{v}^{(n)}_i= \lambda^{(n)}_i \bm{v}^{(n)}_i.  
	\end{split}
\end{equation}
From the second equation in \eqref{19070901}, we have 
\begin{align}
	\bm{v}^{(n)}_i=  u^{(n)}_i \big(\lambda^{(n)}_i-H^{(n-1)} \big)^{-1} \bm{a}^{(n)}.  \label{19070902}
\end{align}
Plugging \eqref{19070902} into the first equation in \eqref{19070901}, we also have\footnote{We point out that $\lambda^{(n)}_i$ might be an eigenvalue of $H^{(n-1)}$, and hence $\lambda^{(n)}_i-H^{(n-1)}$ is not invertible and $u_i^{(n)} = 0$. For continuously distributed matrix entries $h_{ij}$ this, however, happens with probability zero. In case of discretely distributed $h_{ij}$ we can simply smooth out the distribution on an exponentially (in $N$) small scale and remove this change in distribution by trivial perturbation theory. We thus ignore this minor technicality in the following.}
\begin{align*}
	h^{(n)}_{11} +  (\bm{a}^{(n)})^*\big(\lambda^{(n)}_i-H^{(n-1)} \big)^{-1} \bm{a}^{(n)}= \lambda^{(n)}_i
\end{align*}
which can be rewritten as 
\begin{align}
	\lambda^{(n)}_i=  h^{(n)}_{11} +  \frac{1}{N}\sum_{\alpha=1}^{n-1} \frac{|\xi_\alpha^{(n)}|^2}{\lambda_i^{(n)}-\lambda_\alpha^{(n-1)}}, \label{19071002}
\end{align}
where 
\begin{align}
	\xi_\alpha^{(n)}:= \sqrt{N}(\bm{w}^{(n-1)}_\alpha)^*\bm{a}^{(n)}.  \label{def of xi}
\end{align}
Using the delocalization of eigenvectors, $\Vert \bm w_{\alpha}^{(n-1)} \Vert_\infty \prec N^{-1/2}$ and using independence of the matrix entries, we easily see that $|\xi_{\alpha}^{(n)}|^2$ is asymptotically
$\chi^2$-distributed, more precisely it is
$\chi^2(2)/2$ (complex case) or $\chi^2(1)$ (real case). We will, however, not need this precise information. We  also have $|\xi_{\alpha}^{(n)}|\prec 1$
using the moment bounds \eqref{eq:momass} on the matrix elements. 
As we are considering the subcritical regime, we will work with $n\in [N-k,N]$ only in this section.

We need the following results on rigidity and level repulsion of the eigenvalues.  We refer to \cite[Theorem~7.6]{loc_sc_gen} for rigidity (which, in a weaker form, has already been used in \eqref{eq:rigid}), and \cite[Proposition~B.17]{benigni2022optimal} for the level repulsion, which can also be obtained by combining \cite[Theorem~2.7]{bourgade2014edge} with  \cite[Remark~1.5]{knowles2013eigenvector}. 
\begin{lemma}[Rigidity and level repulsion] \label{lem.011301} For any  $n\in [N-k,N]$, we have the following statements.
	\noindent(i) \textnormal{[Rigidity]} For any $i\in [n]$,
	\begin{align*}
		|\lambda_i^{(n)}-\gamma_i^{(n)}| \prec \min \{i, n-i+1\}^{-1/3} n^{-2/3},
	\end{align*} 
	where $\gamma_i^{(n)}$ is defined in \eqref{eq:quantiles}. 
	
	\noindent(ii) \textnormal{[Level repulsion]} For any fixed $i$, and for any sufficiently small $\delta>0$, there exists a small $a>0$, such that  
	\begin{align*}
		\mathbf{P}\big(|\lambda_i^{(n)}-\lambda_{i+1}^{(n)}|< N^{-2/3-\delta}\big)\leq N^{-a}. 
	\end{align*}
\end{lemma}
Finally, we denote 
the event of level repulsion at the $i$-th gap of $H^{(n)}$ as
\begin{align}
	\mathcal{E}_i^{(n)}(\delta)=\{|\lambda_i^{(n)}-\lambda_{i+1}^{(n)}|\geq N^{-2/3-\delta}\}. \label{011310}
\end{align}
We will choose $\delta\equiv \delta(\varepsilon)$ to be sufficiently small according to $\varepsilon$ in the upper bound $k\leq N^{2/3-\varepsilon}$. 

We first consider the case $i=1$, and later extend the discussion to generally fixed  $i$ (see Section \ref{subsubsec:extension}).

\subsubsection{One step eigenvalue difference} From \eqref{19071002},  $|h_{11}^{(n)}|\prec N^{-1/2}$  and  rigidity for $\lambda_1^{(n)}$
from  part (i) of Lemma \ref{lem.011301}, we can  get
\begin{align}
	2\sqrt{\frac{n}{N}}=   \frac{1}{N}\frac{|\xi_1^{(n)}|^2}{\lambda_1^{(n)}-\lambda_1^{(n-1)}}+\frac{1}{N}\sum_{\alpha\geq 2}^{n-1} \frac{|\xi_\alpha^{(n)}|^2}{\lambda_1^{(n)}-\lambda_\alpha^{(n-1)}}+O_\prec(N^{-\frac12}). \label{19071006}
\end{align} 
First, by level repulsion and rigidity for $\lambda_1^{(n)}$  from Lemma \ref{lem.011301}, together with the fact $|\xi|\prec 1$, we have 
\begin{align}
	\frac{1}{N}\sum_{\alpha\geq 2}^{n-1} \frac{|\xi_\alpha^{(n)}|^2}{\lambda_1^{(n)}-\lambda_\alpha^{(n-1)}}+\frac{1}{N}\sum_{\alpha=1}^{n-1} \frac{|\xi_\alpha^{(n)}|^2}{\lambda_\alpha^{(n-1)}-2\sqrt{\frac{n}{N}}-\ii N^{-\frac23+\delta}}= O_\prec(N^{-\frac13+C\delta}), \quad \text{on} \quad  \mathcal{E}_1^{(n)}(\delta).  \label{19071003}
\end{align}
By the local law for Green function entries (see, e.g., \cite[Theorem 6.7]{erdHos2017dynamical}) and the fact $n\in [N-k,N]$, we have 
\begin{align*}
	G^{(n)}_{ii}(2\sqrt{\frac{n}{N}}+\ii N^{-\frac23+\delta})=\sqrt{\frac{N}{n}} m_{\mathrm{sc}}\Big(2+\ii \sqrt{\frac{N}{n}}N^{-\frac23+\delta}\Big)+O_\prec(N^{-\frac13+C\delta})=-1+O_\prec(N^{-\frac13+C\delta}),
\end{align*}
where $G^{(n)}$ is the Green function (resolvent) of $H^{(n)}$. Consequently, by Schur complement, 
we also have 
\begin{align}
	\frac{1}{N}\sum_{\alpha=1}^{n-1} \frac{|\xi_\alpha^{(n)}|^2}{\lambda_\alpha^{(n-1)}-2\sqrt{\frac{n}{N}}-\ii N^{-\frac23+\delta}} &=-\Big(G_{11}^{(n)}\Big(2\sqrt{\frac{n}{N}}+\ii N^{-\frac23+\delta}\Big)\Big)^{-1}+h_{11}^{(n)}-2\sqrt{\frac{n}{N}}-\ii N^{-\frac23+\delta} \notag\\
	&= -1+O_\prec(N^{-\frac13+C\delta}).  \label{19071004}
\end{align}
Combining \eqref{19071003} with \eqref{19071004},  we obtain 
\begin{align}
	\frac{1}{N}\sum_{\alpha\geq 2}^{n-1} \frac{|\xi_\alpha^{(n)}|^2}{\lambda_1^{(n)}-\lambda_\alpha^{(n-1)}}= 1+O_\prec(N^{-\frac13+C\delta}), \quad \text{on} \quad  \mathcal{E}_1^{(n)}(\delta). \label{19071005}
\end{align}
Plugging \eqref{19071005} into \eqref{19071006} we get 
\begin{align}
	N(\lambda_1^{(n)}-\lambda_1^{(n-1)})=|\xi_1^{(n)}|^2+ O_\prec(N^{-\frac13+C\delta}), \quad \text{on} \quad  \mathcal{E}_1^{(n)}(\delta). \label{19071007}
\end{align}
This shows that the leading order fluctuation of $\lambda_1^{(n)}-\lambda_1^{(n-1)}$ comes from the asymptotically 
$\chi^2$ distributed random  variable $|\xi_1^{(n)}|^2$. 

\subsubsection{Iterating \eqref{19071007}: Maintaining the level repulsion event}The above analysis works for one fixed $n\in [N-k,N]$, on  the event $\mathcal{E}_1^{(n)}(\delta)$, which holds with probability at least $1-N^{-a}$ according to Lemma~\ref{lem.011301}. In order to get an estimate of $\lambda_1^{(N)}-\lambda_1^{(N-k)}$ which can be written as the telescopic
sum of $k$ one step differences, we will need the validity of the event $\mathcal{E}_1^{(n)}(\delta)$
simultaneously  for all $n\in [N-k, N]$. A naive union estimate of the probability bound
on $\mathcal{E}_1^{(n)}(\delta)$  in Lemma \ref{lem.011301} (ii) is not enough.  Our strategy is to show that the level repulsion event, $\mathcal{E}_1^{(n)}(\delta)$, once is
assumed on level $n=N$, can be maintained all the way down to the level $n=N-k$, with slight modifications on $\delta$, if necessary. We start with $n=N$. 

To avoid confusion we restate  \eqref{19071007} in this case as  follows
\begin{align}
	N(\lambda_1^{(N)}-\lambda_1^{(N-1)})=|\xi_1^{(N)}|^2+O_\prec(N^{-1/3+C\delta}), \quad \text{on} \quad  \mathcal{E}^{(N)}_1(\delta). \label{010602}
\end{align}
Then, we apply \eqref{19071002} with $i=2$ and $n=N$. By Lemma \ref{lem.011301} (i), we can first write 
\begin{align}
	2=\frac{1}{N}\frac{|\xi_2^{(N)}|^2}{\lambda_2^{(N)}-\lambda_2^{(N-1)}}+\frac{1}{N}\frac{|\xi_1^{(N)}|^2}{\lambda_2^{(N)}-\lambda_1^{(N-1)}}+ \frac{1}{N}\sum_{\alpha=3}^{N-1} \frac{|\xi_\alpha^{(N)}|^2}{\lambda_2^{(N)}-\lambda_\alpha^{(N-1)}} +O_\prec(N^{-1/2}). \label{010601}
\end{align}
We can first bound the second term in the right-hand-side~by \eqref{010602} and the definition of $\mathcal{E}^{(N)}_1(\delta)$, and further derive
\begin{align*}
	2 &=\frac{1}{N}\frac{|\xi_2^{(N)}|^2}{\lambda_2^{(N)}-\lambda_2^{(N-1)}}+ \frac{1}{N}\sum_{\alpha=3}^{N-1} \frac{|\xi_\alpha^{(N)}|^2}{\lambda_2^{(N)}-\lambda_\alpha^{(N-1)}} +O_\prec(N^{-1/3+C\delta})\notag\\
	&\geq \frac{1}{N}\frac{|\xi_2^{(N)}|^2}{\lambda_2^{(N)}-\lambda_2^{(N-1)}}+ \frac{1}{N}\sum_{\alpha=N^{\varepsilon}}^{N-1} \frac{|\xi_\alpha^{(N)}|^2}{\lambda_2^{(N)}-\lambda_\alpha^{(N-1)}} +O_\prec(N^{-1/3+C\delta})\notag\\
	&\geq  \frac{1}{N}\frac{|\xi_2^{(N)}|^2}{\lambda_2^{(N)}-\lambda_2^{(N-1)}}+1 +O_\prec(N^{-1/3+C\delta}),\quad  \text{on} \quad  \mathcal{E}_1^{(N)}(\delta).
\end{align*}
where in the last step we used local law, similarly to the derivation in \eqref{19071003}--\eqref{19071005}.  But this time we do not need a level repulsion between $\lambda_2^{(N)}$ and $\lambda_3^{(N)}$. 
Consequently, we obtain
\begin{align}
	N(\lambda_2^{(N)}-\lambda_2^{(N-1)})\geq |\xi_2^{(N)}|^2 \big(1+O_\prec(N^{-1/3+C\delta})\big),\quad \text{on} \quad  \mathcal{E}_1^{(N)}(\delta), \label{010605}
\end{align}
Combining \eqref{010602} with \eqref{010605}, we obtain
\begin{align}
	\lambda_1^{(N-1)}-\lambda_2^{(N-1)}\geq \lambda_1^{(N)}-\lambda_2^{(N)}+\frac{1}{N}\Big(|\xi_2^{(N)}|^2-|\xi_1^{(N)}|^2\Big)+O_\prec(N^{-4/3+C\delta})\notag\\
	\geq N^{-2/3-\delta}+\frac{1}{N}\Big(|\xi_2^{(N)}|^2-|\xi_1^{(N)}|^2\Big)+O_\prec(N^{-4/3+C\delta}) ,\quad \text{on} \quad  \mathcal{E}_1^{(N)}(\delta) \label{012021}
\end{align}
where in the last estimate we used the definition in \eqref{011310}. 

The rest of the proof is divided in three steps. 
\subsubsection{Step 1: Modified level repulsion events}
The goal of this section, is to show that, for sufficiently small $\delta \equiv \delta(\epsilon) >0$, there exists $a > 0$ such that 
\begin{align}
	\mathbf{P}\big( \mathcal{E}_1(\delta)\big)\geq 1-N^{-a},\quad \text{where} \quad \mathcal{E}_1(\delta):=\bigcap_{j=0}^k \mathcal{E}_{1}^{(N-j)}(\delta). \label{011615}
\end{align}

 In order to do so and obtain level repulsion for the $(N-\ell)$-th minors,
we introduce a sequence of modified level repulsion events 
\begin{align}
	\widehat{\mathcal{E}}_1^{(N-\ell)}(\delta):=\left\{|\lambda_1^{(N-\ell)}-\lambda_{2}^{(N-\ell)}|>N^{-2/3-\delta}+\frac{1}{N}\sum_{j=0}^{\ell-1}\Big(|\xi_2^{(N-j)}|^2-|\xi_1^{(N-j)}|^2\Big)-4\ell N^{-4/3+2C\delta}\right\}, \label{012050}
\end{align}
and we further introduce the event 
\begin{align*}
	\widetilde{\mathcal{E}}_1(\delta):=\left\{\max_{\ell\in [k]} \Big|\sum_{j=0}^{\ell-1}\Big(|\xi_2^{(N-j)}|^2-|\xi_1^{(N-j)}|^2\Big)\Big|\leq \sqrt{k} N^{-1+\delta}\right\}.
\end{align*}
We shall always choose  $\delta$ to be sufficiently small  according to $\varepsilon$ in $k\leq N^{2/3-\varepsilon}$ so that $\sqrt{k} N^{-1+\delta}\ll N^{-2/3-\delta}$. Note that when $\ell=0$, we have $\widehat{\mathcal{E}}_1^{(N)}(\delta)=\mathcal{E}_1^{(N)}(\delta)$. Hence, we can rewrite \eqref{012021} as
\begin{align*}
	\mathbf{P}\Big( \widehat{\mathcal{E}}_1^{(N)}(\delta)\cap \widehat{\mathcal{E}}_1^{(N-1)}(\delta)\Big)\geq \mathbf{P}\Big( \widehat{\mathcal{E}}_1^{(N)}(\delta)\Big)-N^{-D}
\end{align*}
for any large constant $D>0$. It further implies
\begin{align*}
	\mathbf{P}\Big( \widetilde{\mathcal{E}}_1(\delta)\cap \widehat{\mathcal{E}}_1^{(N)}(\delta)\cap \widehat{\mathcal{E}}_1^{(N-1)}(\delta)\Big)\geq \mathbf{P}\Big( \widetilde{\mathcal{E}}_1(\delta)\cap \widehat{\mathcal{E}}_1^{(N)}(\delta)\Big)-N^{-D}.
\end{align*}
Similarly to the proof of \eqref{012021}, we can prove  for any $\ell \le k $ that 
\begin{align}
	\lambda_1^{(N-\ell)}-\lambda_2^{(N-\ell )}\geq  \lambda_1^{(N-\ell+1)}-\lambda_2^{(N-\ell+1)}+\frac{1}{N}\Big(|\xi_2^{(N-\ell+1)}|^2-|\xi_1^{(N-\ell+1)}|^2\Big)+O_\prec( N^{-4/3+C\delta}), \notag\\
	\text{on} \quad  \widetilde{\mathcal{E}}_1(\delta)\cap \widehat{\mathcal{E}}_1^{(N-\ell+1)}(\delta), \label{012030}
\end{align}
as on $\widetilde{\mathcal{E}}_1(\delta)\cap \widehat{\mathcal{E}}_1^{(N-\ell+1)}(\delta)$ we have sufficiently nice level repulsion for $\lambda_1^{(N-\ell+1)}-\lambda_2^{(N-\ell+1)}$. 

 Summing up these inequalities telescopically back to level $N$, we have 
\begin{align*}
	\lambda_1^{(N-\ell)}-\lambda_2^{(N-\ell )} &\geq  \lambda_1^{(N)}-\lambda_2^{(N)} +\sum_{j=0}^{\ell-1}\frac{1}{N}\Big(|\xi_2^{(N-j)}|^2-|\xi_1^{(N-j)}|^2\Big)+O_\prec(\ell N^{-4/3+C\delta}), \notag\\
	&\geq N^{-2/3-\delta}+\sum_{j=0}^{\ell}\frac{1}{N}\Big(|\xi_2^{(N-j)}|^2-|\xi_1^{(N-j)}|^2\Big)+O_\prec(\ell N^{-4/3+C\delta})
	\notag\\ 
	&\hspace{35ex}\text{on} \quad  \widetilde{\mathcal{E}}_1(\delta)\cap \bigcap_{j=0}^{\ell-1} \widehat{\mathcal{E}}_1^{(N-j)}(\delta),
\end{align*}
which implies that 
\begin{align*}
	\mathbf{P}\Big( \widetilde{\mathcal{E}}_1(\delta)\cap \bigcap_{j=0}^{\ell} \widehat{\mathcal{E}}_1^{(N-j)}(\delta)\Big)\geq \mathbf{P}\Big( \widetilde{\mathcal{E}}_1(\delta)\cap \bigcap_{j=0}^{\ell-1} \widehat{\mathcal{E}}_1^{(N-j)}(\delta)\Big)-N^{-D}.
\end{align*}
Setting $\ell=k$ and iterating the above inequality in $j$, we can finally arrive at 
\begin{align}
	\mathbf{P}\Big( \widetilde{\mathcal{E}}_1(\delta)\cap \bigcap_{j=0}^{k} \widehat{\mathcal{E}}_1^{(N-j)}(\delta)\Big)\geq \mathbf{P}\Big( \widetilde{\mathcal{E}}_1(\delta)\cap \mathcal{E}_1^{(N)}(\delta)\Big)-k N^{-D} \label{012060}
\end{align}
by the fact $\widehat{\mathcal{E}}_1^{(N)}(\delta)=\mathcal{E}_1^{(N)}(\delta)$. 
We also claim that 
\begin{align}
	\mathbf{P} \Big(\widetilde{\mathcal{E}}_1(\delta)\Big)\geq 1-N^{-2\delta}, \label{012061}
\end{align}
which will be proved later in Section \ref{subsubsec:hatevent}.  Then, by the definition in \eqref{012050}, we notice that 
\begin{align}
	\widetilde{\mathcal{E}}_1(\delta)\cap \bigcap_{j=0}^{k} \widehat{\mathcal{E}}_1^{(N-j)}(\delta)\subseteq  \bigcap_{j=0}^{k} \mathcal{E}_1^{(N-j)}(2\delta) \label{012062}
\end{align}
Hence, combining \eqref{012060}--\eqref{012062} with Lemma \ref{lem.011301} (ii) with $i=1$, we conclude
\begin{align*}
	\mathbf{P}\Big( \bigcap_{j=0}^{k} \mathcal{E}_1^{(N-j)}(2\delta) \Big)\geq 1-N^{-a}-N^{-2\delta}-kN^{-D}. 
\end{align*}
By redefining $2\delta$ as $\delta$ and modify the value of $a$ slightly, we  can conclude \eqref{011615}.

\subsubsection{Step 2: Maximum inequality for martingales} \label{subsubsec:hatevent}In this second step, we prove the claim \eqref{012061}, by the maximum inequality for martingale. To this end, we introduce the filtration 
\begin{align*}
	\mathcal{F}_{N,s}:=\sigma\{H^{(N-k+s)}\}, \qquad s=0, \ldots, k
\end{align*}
and set
\begin{align*}
	Z_{N,s}:=|\xi_2^{(N-k+s)}|^2-|\xi_1^{(N-k+s)}|^2, \qquad X_{N,m}:=\sum_{s=1}^{m}Z_{N,s} \,. 
\end{align*}
According to the definition in \eqref{def of xi}, it is easy to check via the independence between $\bm{a}^{(N-k+s)}$ and $\mathcal{F}_{N, s-1}$
that $\{Z_{N,s}\}$ is a sequence of martingale difference adapted to $\{\mathcal{F}_{N,s}\}$ as
\begin{align*}
	\mathbf{E}\big(Z_{N,s}|\mathcal{F}_{N,s-1}\big)=0. 
\end{align*} 
We then further write 
\begin{align}
	\sum_{j=0}^{\ell-1}\Big(|\xi_2^{(N-j)}|^2-|\xi_1^{(N-j)}|^2\Big)=\sum_{s=1}^kZ_{N,s}-\sum_{s=1}^{k-\ell}Z_{N,s}=X_{N,k}-X_{N,k-\ell}. \label{012010}
\end{align}
Hence, we have 
\begin{align}
	\max_{\ell\in [k]} \Big|\sum_{j=0}^{\ell-1}\Big(|\xi_2^{(N-j)}|^2-|\xi_1^{(N-j)}|^2\Big)\Big| &\leq  |X_{N,k}|+\max_{m\in[k]} |X_{N,m}|. \label{012011}
\end{align}
By Doob's maximal inequality for martingale, we have  
\begin{align*}
	\mathbf{P} \Big( \frac{1}{N}\max_{m\in [k]} |X_{N,m}|\geq \sqrt{k} N^{-1+\delta}\Big)\leq \frac{\mathbf{E}|X_{N,k}|^2}{kN^{2\delta}}=O(N^{-2\delta}). 
\end{align*}
We obtain from \eqref{012010} and \eqref{012011} that \eqref{012061} holds.

\subsubsection{Step 3: Telescopic summation and martingale CLT} In the final step,  from summing up \eqref{19071007} telescopically,  we have 
\begin{align}
	N(\lambda_1^{(N)}-\lambda_1^{(N-k)})-k=\sum_{s=1}^k(|\xi_1^{(N-k+s)}|^2-1)+ O(kN^{-\frac13+C\delta})\notag\\
	=:\sum_{s=1}^k Y_{N,s}+ O(kN^{-\frac13+C\delta}), \quad \text{on} \quad \mathcal{E}_1(\delta). \label{100730}
\end{align}
Notice that 
$\{Y_{N,s}\}_{s=1}^k$ is also a sequence of martingale differences adapted to $\{\mathcal{F}_{N,s}\}_{s=0}^{k}$ as 
\begin{align*}
	\mathbf{E}\big( Y_{N,s}|\mathcal{F}_{N,s-1}\big)=0. 
\end{align*} 
Further, by definition and the delocalization of eigenvectors (see, e.g., \cite[Theorem~18.9]{erdHos2017dynamical}), we have 
\begin{align}
	\mathbf{E}\big(Y_{N,s}^2|\mathcal{F}_{N,s-1}\big)=1+O_\prec(\frac{1}{N}). \label{conditional variance}
\end{align}
Here we assumed that $H^{(N)}$ is complex Wigner. In the real Wigner case, the above variance is 2.  Hence, we have 
\begin{align*}
	\sigma_{N,k}^2:=\sum_{s=1}^k\mathbf{E}\big(Y_{N,s}^2|\mathcal{F}_{N,s-1}\big)=k+O_\prec(\frac{k}{N}). 
\end{align*}
Further, for any fixed $N$-independent $\varepsilon>0$
\begin{align*}
	\sum_{s=1}^k\mathbf{E}\big(Y_{N,s}^2\bm{1}(|Y_{N,s}|\geq k\varepsilon)\Big)\leq \frac{1}{(k\varepsilon)^2}\sum_{s=1}^k\mathbf{E} (Y_{N,s}^4)=O\left(\frac{1}{k\varepsilon^2}\right), 
\end{align*}
which goes to $0$ as $k=k(N)$ goes to infty. Hence, by Martingale CLT \cite[Corollary~3.1]{hall2014martingale}, we conclude
\begin{align*}
	\frac{1}{\sqrt{k}}\sum_{i=1}^k Y_{N,i}\Rightarrow \mathcal{N}(0,1). 
\end{align*}
Hence, we derive from \eqref{100730} that 
\begin{align}
	\frac{N(\lambda_1^{(N)}-\lambda_1^{(N-k)})-k}{\sqrt{k}}\Rightarrow \mathcal{N}(0,2/\beta)\,,   \label{CLT}
\end{align}
where $\beta = 1$ for real symmetric and $\beta = 2$ for complex Hermitian Wigner matrices. This gives
 the result in Proposition~\ref{EigenvalueSub} for $i=1$. 

\subsubsection{Extension for $i \ge 2$} \label{subsubsec:extension}Finally, we extend the above discussion to  $\lambda_i^{(N)}-\lambda_i^{(N-k)}$ for any fixed $i$ and get the same conclusion as \eqref{CLT}. For instance, when $i=2$, assuming the level repulsion for both $\lambda_1^{(N)}-\lambda_2^{(N)}$ and $\lambda_2^{(N)}-\lambda_3^{(N)}$, i.e, on event $\mathcal{E}_1^{(N)}(\delta)\cap \mathcal{E}_2^{(N)}(\delta)$, we can get from  \eqref{010601} that 
\begin{align}
	N(\lambda_2^{(N)}-\lambda_2^{(N-1)})= |\xi_2^{(N)}|^2+O_\prec(N^{-1/3+C\delta}),\quad \text{on} \quad \mathcal{E}_1^{(N)}(\delta)\cap \mathcal{E}_2^{(N)}(\delta), \label{011601}
\end{align}
as a refinement of \eqref{010605}. The derivation of \eqref{011601} can be done similarly to \eqref{19071007}. Further, similarly to \eqref{010605}, one can derive 
\begin{align}
	N(\lambda_3^{(N)}-\lambda_3^{(N-1)})\geq |\xi_3^{(N)}|^2 \big(1+O_\prec(N^{-1/3+C\delta})\big),\quad \text{on} \quad  \mathcal{E}_1^{(N)}(\delta)\cap \mathcal{E}_2^{(N)}(\delta). \label{011602}
\end{align}
Combining \eqref{011601} with \eqref{011602} leads to 
\begin{align*}
	\lambda_2^{(N-1)}-\lambda_3^{(N-1)}\geq \lambda_2^{(N)}-\lambda_3^{(N)}+\frac{1}{N}\Big(|\xi_3^{(N)}|^2-|\xi_2^{(N)}|^2\Big)+O_\prec(N^{-4/3+C\delta}), \quad \text{on} \quad  \mathcal{E}_1^{(N)}(\delta)\cap \mathcal{E}_2^{(N)}(\delta).
\end{align*}
Similarly to the proof of \eqref{011615}, we can iterate the above discussion and conclude 
\begin{align*}
	\mathbf{P}\Big( \mathcal{E}_1(\delta)\cap \mathcal{E}_2(\delta)\Big)\geq 1-N^{-a},\qquad \mathcal{E}_\alpha(\delta):=\bigcap_{j=0}^k \mathcal{E}_{\alpha}^{(N-j)}(\delta).
\end{align*}
by adjusting the values of $\delta$ and $a$ slightly. On the event $\mathcal{E}_1(\delta)\cap \mathcal{E}_2(\delta)$, we have \eqref{011601} if $N$ is replaced by any $N-j$ with $j\in [k]$. Summing up all these one step differences and applying the martingale CLT we can conclude the proof for $ \lambda_2^{(N)}-\lambda_2^{(N-k)}$. Further extension to $ \lambda_i^{(N)}-\lambda_i^{(N-k)}$ for any fixed $i$ is analogous. 

One can actually consider the joint distribution of all $ (\lambda_i^{(N)}-\lambda_i^{(N-k)})$'s with fixed $i$. It would be sufficient to study a generic linear combination of $[\sum_{s=1}^k(|\xi_i^{(N-k+s)}|^2-1)]$'s for $1\leq i\leq C$, according to \eqref{100730} and its $i$-analogue. Such a linear combination is again a martingale adapted to $\{\mathcal{F}_{N,s}\}_{s=0}^{k}$. As an extension of \eqref{conditional variance},  we need to compute the conditional variance for the martingale difference, which is now a linear combination of $(|\xi_i^{(N-k+s)}|^2-1)$'s over $i$. By the orthogonality of $\bm{w}^{(n-1)}_\alpha$'s in \eqref{def of xi}, one can easily check that the $(|\xi_i^{(N-k+s)}|^2-1)$'s can be decoupled in the variance calculation of the martingale difference. This leads to the fact that the centered $ (\lambda_i^{(N)}-\lambda_i^{(N-k)})$'s will converge to independent Gaussian variables. 
\nc
\\[2mm]
This concludes the proof of Proposition \ref{EigenvalueSub}. \qed

\subsection{Proof of Proposition \ref{EigenvectorSub}} \label{subsec:evecsub}
In this section, we turn to the eigenvector overlaps and prove Proposition~\ref{EigenvectorSub}. Just as in Section \ref{subsec:evalsub}, we fix an index $i$ which is independent of $N$. We set for any $n_1>n_2,$
\begin{align*}
	\mathcal{O}_{i}^{(n_1,n_2)}:= \langle \bm{w}_i^{(n_1)}, \widetilde{\bm{w}}_i^{(n_2)}\rangle,
\end{align*}
where the augmented vector $\widetilde{\bm{w}}_i^{(n_2)}$ is obtained by adding $n_1-n_2$ zero entries on top of $\bm{w}_i^{(n_2)}$.  A convention of the phases of the eigenvectors will be made in \eqref{19071009}. 
Using \eqref{19070902}, we have
\begin{align*}
	1=\|\bm{v}^{(n)}_i\|^2+|u^{(n)}_i|^2= |u^{(n)}_i|^2\Big(1+(\bm{a}^{(n)})^*\big(\lambda^{(n)}_i-H^{(n-1)} \big)^{-2} \bm{a}^{(n)}\Big),
\end{align*}
and thus
\begin{align}
	|u^{(n)}_i|^2= \frac{1}{1+\frac{1}{N}\sum_{\alpha=1}^{n-1} \frac{|\xi_\alpha^{(n)}|^2}{(\lambda_i^{(n)}-
			\lambda_\alpha^{(n-1)})^2}}. 
\end{align}
We can also write \eqref{19070902} as 
\begin{align}
	\bm{v}^{(n)}_i= \frac{\mathrm{e}^{\ii \arg(u_i^{(n)})}}{\sqrt{1+\frac{1}{N}\sum_{\alpha=1}^{n-1} \frac{|\xi_\alpha^{(n)}|^2}{(\lambda_i^{(n)}-\lambda_\alpha^{(n-1)})^2}}}  \frac{1}{\sqrt{N}}\sum_{\alpha=1}^{n-1} \frac{\xi_\alpha^{(n)}}{\lambda^{(n)}_i-\lambda^{(n-1)}_\alpha} \bm{w}^{(n-1)}_\alpha . \label{19071001}
\end{align}
Multiplying by $\big(\bm{w}^{(n-1)}_i\big)^*$ on both sides, we can get 
\begin{align*}
	\mathcal{O}_i^{(n,n-1)}= \frac{\mathrm{e}^{\ii \arg(u_i^{(n)})}}{\sqrt{1+\frac{1}{N}\sum_{\alpha=1}^{n-1} \frac{|\xi_\alpha^{(n)}|^2}{(\lambda_i^{(n)}-\lambda_\alpha^{(n-1)})^2}}}  \frac{1}{\sqrt{N}} \frac{ \xi_i^{(n)}}{\lambda^{(n)}_i-\lambda^{(n-1)}_i}. 
\end{align*}
Note that, starting from $n=N$, we can successively choose the phase of $\bm{w}_{i}^{(n-1)}$ according to that of 
$\bm{w}_{i}^{(n)}$ so that $\mathrm{e}^{\ii \arg(u_i^{(n)})}\xi_i^{(n)}=|\xi_i^{(n)}|$ for all $n\in [N-k,N]$. 
Under this choice of phases, we have
\begin{align}
	\mathcal{O}_i^{(n,n-1)}= \frac{1}{\sqrt{1+\frac{1}{N}\sum_{\alpha=1}^{n-1} \frac{|\xi_\alpha^{(n)}|^2}{(\lambda_i^{(n)}-\lambda_\alpha^{(n-1)})^2}}}  \frac{1}{\sqrt{N}} \frac{ |\xi_i^{(n)}|}{\lambda^{(n)}_i-\lambda^{(n-1)}_i}.  \label{19071009}
\end{align}
Then by local law and level repulsion in Lemma \ref{lem.011301}, we have for any $n\in [N-k, N]$, 
\begin{align*}
	\frac{1}{N}\sum_{\alpha=1}^{n-1} \frac{|\xi_\alpha^{(n)}|^2}{(\lambda_i^{(n)}-\lambda_\alpha^{(n-1)})^2} &=\frac{1}{N} \frac{|\xi_i^{(n)}|^2}{(\lambda_i^{(n)}-\lambda_i^{(n-1)})^2}+\frac{1}{N}\sum_{\alpha\neq i} \frac{|\xi_\alpha^{(n)}|^2}{(\lambda_i^{(n)}-\lambda_\alpha^{(n-1)})^2}
	\nonumber\\
	& =\frac{1}{N} \frac{|\xi_i^{(n)}|^2}{(\lambda_i^{(n)}-\lambda_i^{(n-1)})^2}+O_\prec(N^{\frac13+C\delta}),  \quad \text{on} \quad  \bigcap_{\ell=1}^i\mathcal{E}_{\ell}^{(n)}(\delta), 
\end{align*}
where we recall \eqref{011310} for the definition of $\mathcal{E}_{\ell}^{(n)}$. 
Plugging the above into \eqref{19071009} we obtain
\begin{align*}
	0\leq 1-\mathcal{O}_i^{(n,n-1)}\prec N^{-\frac23+C\delta}, \quad \text{on} \quad  \bigcap_{\ell=1}^i\mathcal{E}_{\ell}^{(n)}(\delta).
\end{align*}
Hence, we can decompose
\begin{align*}
	\bm{w}_i^{(n)}=\mathcal{O}_i^{(n,n-1)}\widetilde{\bm{w}}_i^{(n-1)}+O_\prec(N^{-\frac23+C\delta}),  \quad \text{on} \quad  \bigcap_{\ell=1}^i\mathcal{E}_{\ell}^{(n)}(\delta).
\end{align*}
where $O_\prec(N^{-\frac23+C\delta})$ represents a vector $ \bm r \in \{\widetilde{\bm{w}}_i^{(n-1)}\}^{\perp}$ with norm $ \Vert \bm r \Vert = O_\prec(N^{-\frac23+C\delta})$. Iterating this process from $N$ to $N-k$, we arrive at 
\begin{align*}
	\bm{w}_i^{(N)}=\prod_{\ell=1}^k\mathcal{O}_i^{(N-\ell+1,N-\ell)}\widetilde{\bm{w}}_i^{(N-k)}+O_\prec(kN^{-\frac23+C\delta}), \quad \text{on} \quad  \bigcap_{\ell=1}^i\mathcal{E}_{\ell}(\delta),
\end{align*}
where again the error term represents a vector with norm bounded by $O_\prec(kN^{-\frac23+C\delta})$. 
This implies 
\begin{align*}
	0\leq1-\mathcal{O}_i^{(N,N-k)}\prec  kN^{-\frac23+C\delta}, \quad \text{on} \quad  \bigcap_{\ell=1}^i\mathcal{E}_{\ell}(\delta)
\end{align*}
and we have thus proven Proposition \ref{EigenvectorSub}.\qed 

\appendix

\section{Two-resolvent local law: Proof of \eqref{eq:2G_special}}
The goal of this appendix  is to provide a proof of the two-resolvent bound in \eqref{eq:2G_special}. This will follow as a special case of a two-resolvent local law for (differently) deformed Wigner matrices (see Appendix \ref{app:2G}), complemented by an analysis of the deterministic approximation to $(H^{(N-k)}-\Re D_k - z_1)^{-1}$ with $D_k$ from \eqref{eq:D_def} (see Appendix~\ref{app:M}). Finally, based on Appendices \ref{app:2G}--\ref{app:M}, the proof of \eqref{eq:2G_special} is given in Appendix~\ref{subsec:2G_special}. 
\subsection{Two-resolvent local law away from the cusp}\label{app:2G}

In this appendix we state and prove a general two-resolvent averaged local law for two differently deformed Wigner matrices. Prior to that we need to introduce some notation. Let $W$ be an $N\times N$ Wigner matrix,\footnote{For better comparability with \cite{eigenv_decorr}, in this appendix, Wigner matrices shall be denoted by the symbol $W$.} $D\in\C^{N\times N}$ a deterministic matrix of the same symmetry type, and $z\in\C\setminus \R$. Following \cite{eigenv_decorr}, we call $\nu:=(z,D)$ a spectral pair and denote $G^D(z):=(W+D-z)^{-1}$. The deterministic approximation $M^D(z)$ to $G^D(z)$ is given by the solution to the \emph{Matrix Dyson Equation} \cite{MDEreview}
\begin{equation}
-(M^D(z))^{-1} = z-D+\langle M^D(z)\rangle,\quad \Im z\Im M^D(z)>0,
\label{eq:MDE}
\end{equation}
and the corresponding density $\rho_D$ is defined by
\begin{equation*}
\rho_D(z):=\pi^{-1}\left\vert\langle\Im M^D(z)\rangle\right\vert,\,z\in\C\setminus\R,\qquad \rho_D(x):=\lim_{\eta\to+0} \rho_D(x+\ii\eta),\ x\in\R.
\end{equation*}
Furthermore, for any fixed $\delta>0$ we define the set of admissible energies as 
\begin{equation*}
\mathbf{D}_\delta(D):= \left\lbrace x\in\R:\, \vert \rho_D(x)-\rho_D(y)\vert\le\delta^{-1}\vert x-y\vert^{1/2},\, \forall y\in\R\right\rbrace.
\end{equation*}
Finally, for two deformations $D_1,D_2\in\C^{N\times N}$ and an observable $B_1\in \C^{N\times N}$ the deterministic approximation to the resolvent chain $G^{D_1}(z_1)B_1G^{D_2}(z_2)$ is given by
\begin{equation*}
M_{\nu_1,\nu_2}^{B_1}:=M^{D_1}(z_1)B_1M^{D_2}(z_2) + \frac{\langle M^{D_1}(z_1)B_1M^{D_2}(z_2)\rangle}{1-\langle M^{D_1}(z_1)M^{D_2}(z_2)\rangle}M^{D_1}(z_1)M^{D_2}(z_2),\quad \nu_l=(z_l,D_l),\,l=1,2.
\end{equation*}
We are now ready to present the main result of this section.
\begin{proposition}[Average two-resolvent local law away from the cusp]\label{prop:2G} Fix constants $L, \delta, \varepsilon>0$. Let $W$ be a Wigner matrix satisfying Assumption \ref{ass:momass}, and let $D_1, D_2\in\C^{N\times N}$ be Hermitian matrices such that $\langle D_l\rangle=0$ and $\lVert D_l\rVert\le L$ for $l=1,2$. Assume that $\lVert M^{D_l}(z)\rVert\le L$ for any $z\in\C\setminus\R$. For spectral parameters $z_1,z_2\in\C\setminus\R$, denote $\eta_l:=\vert\Im z_l\vert$, $\eta := \eta_1 \wedge \eta_2$, $\rho_l:=\rho_{D_l}(z_l)$ and $\ell:=\eta_1\rho_1\wedge\eta_2\rho_2$. Moreover, we introduce the control parameter
\begin{equation}
\gamma:=\langle (D_1-D_2)^2\rangle + \eta_1/\rho_1+\eta_2/\rho_2.
\label{eq:gamma}
\end{equation}

Then for any deterministic $B_1,B_2\in\C^{N\times N}$ we have
\begin{equation}
\left\vert \langle \left(G^{D_1}(z_1)B_1G^{D_2}(z_2)-M_{\nu_1,\nu_2}^{B_1}\right)B_2\rangle\right\vert \prec \left(\frac{1}{\sqrt{N\ell}\gamma} \wedge \frac{1}{N \eta^2}\right)\lVert B_1\rVert\lVert B_2\rVert
\label{eq:2G_loc}
\end{equation}
uniformly in $B_1, B_2$, spectral parameters satisfying $\Re z_l\in\mathbf{D}_\delta(D_l)$, $|z_l|\le N^{100}$, for $l=1,2$, and $\ell\ge N^{-1+\varepsilon}$.
\end{proposition}

This result is an extension of \cite[Theorem 3.2, P.1]{eigenv_decorr} to a broader range of admissible energies. While \cite[Theorem 3.2, P.1]{eigenv_decorr} applies to the \emph{bulk}, i.e.~holds for 
\begin{equation*}
\Re z_l\in \mathbf{B}_\kappa(D_l) = \{x\in\R\,:\, \rho_{D_l}(x)\ge\kappa\}
\end{equation*}
for some $\kappa>0$, Proposition \ref{prop:2G} is valid also near the \emph{regular edge}, where $\rho_{D_l}$ vanishes like a square root (see \cite[Definition 7.17]{shape}). However, by restricting to $\Re z_l\in\mathbf{D}_\delta(D_l)$ we exclude the cusp regime, where $\rho_{D_l}$ vanishes like a cubic root.

We complement Proposition \ref{prop:2G} by the following bound on the deterministic approximation $M_{\nu_1,\nu_2}^{B_1}$, which can immediately be obtained from \cite[Proposition 4.6]{eigenv_decorr} and \cite[Proposition 3.1]{eigenv_decorr}.
\begin{proposition}\label{prop:M2_bound} Under the above assumptions and notations, we have
\begin{equation} \label{eq:M2bound}
\left\lVert M_{\nu_1,\nu_2}^{B_1}\right\rVert\lesssim \gamma^{-1} \Vert B_1 \Vert
\end{equation}
with an implicit constant depending only on $L$.
\end{proposition}

We are now left with providing the proof of Proposition \ref{prop:2G}.

\begin{proof}[Proof of Proposition \ref{prop:2G}] 
	We have already mentioned above that Proposition \ref{prop:2G} is an extension of \cite[Theorem 3.2, P.1]{eigenv_decorr} to the edge regime. While \cite{eigenv_decorr} covered exclusively the bulk, the more recent paper \cite{nonHermdecay} expanded the proof strategy in \cite{eigenv_decorr} (the \emph{Zigzag strategy}) outside of the bulk regime. By the strong analogy of \cite{eigenv_decorr} and \cite{nonHermdecay} and for brevity of the argument, we shall henceforth restrict ourselves to explaining how the proof strategy of \cite{nonHermdecay}, adapted from \cite{eigenv_decorr}, can be applied to our current setting. %, in particular allowing 
	
In fact, the direct analog of \eqref{eq:2G_loc} in \cite{nonHermdecay} is given by \cite[Eq.~(3.26)]{nonHermdecay} with our $\gamma$ in \eqref{eq:gamma} playing the same role as $\gamma$ in \cite[Eq.~(3.23)]{nonHermdecay}. By inspecting the proof of \cite[Eq.~(3.26)]{nonHermdecay}, one finds that the key inputs for the argument are certain bounds on the deterministic approximations to two- and three-resolvent chains as formulated in \cite[Eqs.~(3.20)+(3.22)]{nonHermdecay}. The analog of \cite[Eq.~(3.20)]{nonHermdecay} is the bound \eqref{eq:M2bound} provided in Proposition~\ref{prop:M2_bound}. Hence, the key step left for transferring the arguments from \cite{nonHermdecay} to our setting is the bound analogous to \cite[Eq.~(3.22)]{nonHermdecay}. This is an improved (compared to \cite{eigenv_decorr}) upper bound on the deterministic approximation to $G^{D_1}(z_1) B_1G^{D_2}(z_2)B_2G^{D_1}(z_1)$ for deterministic matrices $B_1, B_2 \in \C^{N \times N}$,  which we denote by $M^{B_1,B_2}_{\nu_1,\nu_2,\nu_1}$. 

For the analog of \cite[Eq.~(3.22)]{nonHermdecay}, we claim that\footnote{In fact, even the slightly stronger bound with $\gamma$ replaced by $\beta_* \gtrsim \gamma$ with $\beta_*$ defined in \cite[Eq.~(3.2)]{eigenv_decorr} holds true. This is relevant when revisiting the proof of a weakened version of the bound \eqref{eq:M3_bound} in \cite{eigenv_decorr}.}
\begin{equation}
\lVert M^{B_1,B_2}_{\nu_1,\nu_2,\nu_1}\rVert + \lVert M^{B_1,B_2}_{\nu_1,\nu_2,\bar{\nu}_1}\rVert\lesssim \frac{\lVert B_1\rVert\lVert B_2\rVert}{\eta \, \gamma}.
\label{eq:M3_bound}
\end{equation}
uniformly for spectral parameters $z_l$, $l =1,2$ whose real parts $\Re z_l$ are admissible energies, $\Re z_l \in \mathbf{D}_\delta(D_l)$, $l=1,2$. In \cite{eigenv_decorr}, the weaker bound \cite[Eq. (4.21b)]{eigenv_decorr} served the role of \eqref{eq:M3_bound}, differing from it by replacing the factor $\eta$ with the smaller factor $\ell$. We prove \eqref{eq:M3_bound} by revisiting the proof of \cite[Eq. (4.21b)]{eigenv_decorr}. In the original argument, the bound was derived uniformly in $z_1,z_2\in\C\setminus\R$ by using $\vert 1-\langle M_1^2\rangle\vert\gtrsim \rho_{D_1}^2$. Since currently $\Re z_1$ is away from the cusp, this estimate improves to $\vert 1-\langle M_1^2\rangle\vert\gtrsim \rho_{D_1}$ by \cite[Corollary 5.3]{shape}. This refinement is sufficient to recover the factor $\eta$ in the right-hand-side of \eqref{eq:M3_bound}. 
	
Armed with \eqref{eq:M2bound}--\eqref{eq:M3_bound}, we can now follow the arguments in \cite{nonHermdecay} for the proof of \cite[Eq.~(3.26)]{nonHermdecay} (see, in particular, \cite[Section~4 (apart from Section 4.1) and Sections~5.1--5.2]{nonHermdecay}), in order to obtain \eqref{eq:2G_loc}. This concludes our sketch of the proof of Proposition \ref{prop:2G}. 
\end{proof}

\subsection{Properties of the deterministic approximation}\label{app:M}

In this section, we prove several properties of the deterministic approximation of $(H^{(N-k)}-\Re D_k - z_1)^{-1}$ with $D_k$ from \eqref{eq:D_def}. Throughout this section we fix a (small) constant $\xi_0>0$ and take $z_1:=E_1+\ii\eta_1$ with $E_1\in\R$ and $\eta_1= N^{-2/3+\xi_0}$ in \eqref{eq:D_def}. Moreover, we stress the dependence of $D_k$ on $E_1$ by using the following notation
\begin{equation}
D_k(E_1):=X_k (H^{[k]}-E_1-\ii\eta_1)^{-1}X_k^*.
\label{eq:D_E_def}
\end{equation} 
Next we condition on $X_k$, $H^{[k]}$ and set $M_{E_1}=M_{E_1}(z)$ to be the deterministic approximation to $(H^{(N-k)}-\Re D_k(E_1) -z)^{-1}$, i.e. $M_{E_1}(z)$ is the solution to the MDE
\begin{equation}
-M_{E_1}^{-1}=z+\Re D_k(E_1) + N^{-1}\mathrm{Tr} M_{E_1}, \quad \Im z\Im M_{E_1}(z)>0.
\label{eq:MDE_k}
\end{equation}
Note that in \eqref{eq:2G_goal} we consider only the case $z:=z_1$, however in the current section we let $z$ to be any complex number with $\Im z\neq 0$.
\begin{lemma}[Preliminary properties of $M_{E_1}$]\label{lem:M_prop} Fix $\delta>0$ and take $z_1=E_1+\ii\eta_1\in\mathbf{H}$ with $\vert E_1\vert > 1+\delta$. The following results hold with very high probability wrt. randomness in $X_k$ and $H^{[k]}$.
\begin{enumerate}
\item For any $z\in \C\setminus\R$ we have that
\begin{equation}
\lVert M_{E_1}(z)\rVert \le \frac{2}{\vert E_1\vert -1}.
\label{eq:M_bound}
\end{equation}
\item Let $\rho_{E_1}$ be the scDoS corresponding to $M_{E_1}$. There exists a constant $C>0$ which depends only on $\delta$ such that
\begin{equation}
[-2+Ck/N,2-Ck/N]\subset \mathrm{supp}(\rho_{E_1})\subset [-2-Ck/N,2+Ck/N].
\label{eq:supp}
\end{equation}
Moreover, $\mathrm{supp}(\rho_{E_1})$ consists of a single interval.

\item Let $\mathrm{E}_+(E_1)$ be the right edge of $\mathrm{supp}\left(\rho_{E_1}\right)$. Then the function $\mathrm{E}_+$ is differentiable on $[1+\delta,2]$ and
\begin{equation}
\left\vert\frac{\dif}{\dif E_1} \mathrm{E}_+(E_1)\right\vert\lesssim \frac{k}{N}
\label{eq:dif_E}
\end{equation}
with implicit constants depending only on $\delta$ and $\xi_0$.
\end{enumerate}
\end{lemma}

\begin{proposition}\label{prop:edge} Let $\varepsilon_0>0$ be the constant from Theorem \ref{theo:large_k}, i.e. $k\le N^{1-\varepsilon_0}$. Then we have
\begin{equation}
\mathrm{E}_+(E_1)=2+\mathcal{O}(N^{-2/3+2\xi_0}), \quad \forall E_1\in [2-N^{-2/3+\varepsilon_0},2]
\label{eq:E_loc}
\end{equation}
with very high probability as $N \to \infty$. 
\end{proposition}

\begin{proof}[Proof of Lemma \ref{lem:M_prop}]
The proof is divided in three parts. \\[1mm]
\textit{Part (1):} First of all we write $M_{E_1}(z)$ as
\begin{equation}
	M_{E_1}(z)=-(\Re D_k(E_1) + z+N^{-1}\mathrm{Tr} M_{E_1}(z))^{-1},
	\label{eq:M_res}
\end{equation}
i.e. $M_{E_1}(z)$ is the resolvent of $-\Re D_k(E_1)$ at the point
\begin{equation}
	\omega = \omega_{E_1}(z):=z+N^{-1}\mathrm{Tr} M_{E_1}(z).
	\label{eq:omega_def}
\end{equation}
By multiplying \eqref{eq:MDE_k} by $M_{E_1}$ and taking the trace we get that 
\begin{equation*}
	\omega = - \frac{N^{-1}(N-k) + N^{-1}\mathrm{Tr} \left[ M_{E_1}(z)\Re D_k(E_1)\right]}{N^{-1}\mathrm{Tr} M_{E_1}(z)}. 
\end{equation*}
Taking imaginary part of \eqref{eq:MDE_k} we conclude that $(N-k)^{-1}\mathrm{Tr}\left[ M_{E_1}(z)M_{E_1}^*(z)\right] \le 1$ and arrive at
\begin{equation}
	|\omega| \ge N^{-1}(N-k) - \left( N^{-1}\mathrm{Tr} \left[ (\Re D_k(E_1))^2\right]\right)^{1/2} = 1-\mathcal{O}\left( \sqrt{k/N}\right),
	\label{eq:omega_bound}
\end{equation}
where in the last step we used \eqref{eq:D_HS} from Lemma \ref{lem:D}. Together \eqref{eq:M_res}, \eqref{eq:omega_bound} and the first part of \eqref{eq:D_bounds} imply~\eqref{eq:M_bound}.\\[1 mm]

\noindent\textit{Part (2):} From \eqref{eq:MDE_k} we see that $\omega$ defined in \eqref{eq:omega_def} satisfies the following quadratic equation:
\begin{equation}
	\omega + \frac{1}{N} \mathrm{Tr} \left(\Re D_k(E_1) +\omega\right)^{-1} = z.
	\label{eq:omega}
\end{equation} 
The explicit formula \eqref{eq:D_E_def} implies that $\mathrm{rank} \, \Re D_k(E_1)\le k$. Therefore, $\Re D_k(E_1)$ has at least $N-2k$ zero eigenvalues. Denote the remaining ones by $\{ \mu_j\}_{j=1}^k$ (some of these $\mu_j$'s may be equal to zero). Using this notation we rewrite \eqref{eq:omega} as
\begin{equation}
	\omega + \frac{1}{\omega} + \frac{1}{N} \left( \sum_{j=1}^k \frac{1}{\mu_j+\omega} -\frac{2k}{\omega}\right)=z.
	\label{eq:omega_fin}
\end{equation}
From \eqref{eq:D_bounds} we have the following bound on $\mu_j$'s:
\begin{equation}
	\vert \mu_j\vert \le \lVert \Re D_k(E_1)\rVert \le \vert E_1\vert^{-1}(1+c),\quad \forall j\in[k],
	\label{eq:mu_bounds}
\end{equation}
for any fixed constant $c>0$. Now we take $z:=E+\ii 0$, $E\in\R$. Note that $E$ is an internal point of $\mathrm{supp}(\rho_{E_1})$ if and only if $\Im \omega>0$. 

Assume that $E$ is outside of $\mathrm{supp}(\rho_{E_1})$. Then $\omega\in\R$ and we have from \eqref{eq:omega_fin} that
\begin{equation*}
	\vert E\vert = \left\vert\omega +\frac{1}{\omega}\right\vert + \mathcal{O}\left(\frac{k}{N}\right)\ge 2 + \mathcal{O}\left(\frac{k}{N}\right),
\end{equation*}
where in the first equality we used \eqref{eq:omega_bound} and \eqref{eq:mu_bounds}. This verifies the first inclusion in \eqref{eq:supp}.

Now consider the case when $E$ is an internal point of $\mathrm{supp}(\rho_{E_1})$, i.e. $\Im \omega>0$. By taking the imaginary part of \eqref{eq:omega_fin} and dividing the result by $\Im\omega$ we get
\begin{equation}
	1-\frac{1}{\vert\omega\vert^2} = \frac{1}{N}\left( \sum_{j=1}^k \frac{1}{\vert\mu_j+\omega\vert^2} - \frac{2k}{|\omega|^2}\right) = \mathcal{O}\left(\frac{k}{N}\right),
	\label{eq:Im_eq_omega}
\end{equation}
where in the second equality we have used \eqref{eq:omega_bound} and \eqref{eq:mu_bounds}. Thus $\vert\omega\vert^2 = 1+\mathcal{O}(k/N)$. Next we take the real part of \eqref{eq:omega_fin} yielding
\begin{equation*}
	\Re\omega \left(1+\frac{1}{|\omega|^2}\right) = E - \frac{1}{N}\left( \sum_{j=1}^k \frac{\mu_j+\Re\omega}{|\mu_j+\omega|^2}-\frac{2k\Re\omega}{\vert\omega\vert^2}\right) = E+\mathcal{O}\left(\frac{k}{N}\right).
\end{equation*}
Since $\vert \Re\omega\vert\le\vert\omega\vert = 1+\mathcal{O}(k/N)$, we conclude that $\vert E\vert\le 2+\mathcal{O}(k/N)$, which finishes the proof of the second inclusion in \eqref{eq:supp}.

Finally, \cite[Theorem 7.1]{shape} asserts that all support intervals of $\rho_{E_1}$ have length of order one. Together with \eqref{eq:supp} this implies that $\rho_{E_1}$ has a single interval of support.\\[1 mm]
\noindent\textit{Part (3):} From \eqref{eq:Im_eq_omega} we see that $\omega_+(E_1):=\omega_{E_1}(\mathrm{E}_+(E_1)+\ii 0)\in\R$ solves the equation
\begin{equation}
	N^{-1}\left\langle \vert \Re D_k(E_1)+\omega_+(E_1)\vert^{-2}\right\rangle = 1.
	\label{eq:omega_+}
\end{equation}
The implicit function theorem implies that $\omega_+(E_1)$ is differentiable on $[1+\delta,2]$ and by differentiating \eqref{eq:omega_+} in $E_1$ we get 
\begin{equation}
	\partial_{E_1}\omega_+(E_1) = - \frac{N^{-1}\mathrm{Tr} \left[ \left(\Re D_k(E_1)+\omega_+(E_1)\right)^{-3}\partial_{E_1}\Re D_k(E_1)\right]}{N^{-1}\mathrm{Tr} \left[ \left(\Re D_k(E_1)+\omega_+(E_1)\right)^{-3}\right]}.
	\label{eq:omega_+_dif}
\end{equation}
Since the distance from $-\omega_+(E_1)$ to the spectrum of $\Re D_k(E_1)$ is of order one (see the proof of the first part of Lemma \ref{lem:M_prop}) and the Hilbert-Schmidt norm of 
\begin{equation*}
	\partial_{E_1}\Re D_k(E_1) =\Re\left( X_k(H^{[k]}-E_1-\ii\eta_1)^{-2}X_k^*\right)
\end{equation*}
is of order $k/N$, we conclude that 
\begin{equation}
	\vert\partial_{E_1}\omega_+(E_1)\vert\lesssim k/N.
	\label{eq:omega_dif_res}
\end{equation}
Consider \eqref{eq:omega} for $z=\mathrm{E}_+(E_1)+\ii 0$ and differentiate this identity in $E_1$. Arguing similarly to \eqref{eq:omega_+_dif} and using \eqref{eq:omega_dif_res} as an input, we finish the proof of \eqref{eq:dif_E}.
\end{proof}

\begin{proof}[Proof of Proposition \ref{prop:edge}]
Taking into account \eqref{eq:dif_E}, we see that it is sufficient to prove \eqref{eq:E_loc} for $E_1=2$. In order to approach this case, observe that the following identity holds:
\begin{equation}
	\begin{split}
		\left(\Im (H-2-\ii\eta_1)^{-1}\right)_{NN} = &\left(\Im(H^{(N-k)}-D_k(E_1)-2-\ii\eta_1)^{-1}\right)_{N-k,N-k}\\
		\sim &\left(\Im(H^{(N-k)}-\Re D_k(2)-2-\ii\eta_1)^{-1}\right)_{N-k,N-k},
	\end{split}
	\label{eq:id_for_edge}
\end{equation}
Here in the first line we used Schur decomposition performed in the same way as in \eqref{eq:2G_goal}, and in the second line applied Lemma \ref{lem:Im}. Using the single-resolvent isotropic local law \cite[Theorem 2.6]{Alt_band} we get
\begin{equation*}
	\left(\Im (H-2-\ii\eta_1)^{-1}\right)_{NN} = \rho_{\rm sc}(2+\ii\eta_1)(1+o(1))\sim \sqrt{\eta_1}
\end{equation*}
with very high probablity. Therefore,
\begin{equation}
	\left(\Im(H^{(N-k)}-\Re D_k(2)-2-\ii\eta_1)^{-1}\right)_{N-k,N-k} \sim \sqrt{\eta_1}.
	\label{eq:edge_ub_aux}
\end{equation}
Next we condition in \eqref{eq:edge_ub_aux} on $D_k(2)$. Assume that \eqref{eq:E_loc} does not hold for the right edge of  the resulting $\rho_2=\rho_{E_1=2}$. The local law \cite[Theorem 2.6]{Alt_band} for the left-hand-side of \eqref{eq:edge_ub_aux} gives that
\begin{equation}
	\left(\Im(H^{(N-k)}-\Re D_k(2)-2-\ii\eta_1)^{-1}\right)_{N-k,N-k} = \rho_2(2+\ii\eta_1) + \mathcal{O}_\prec\left(\frac{1}{N\eta_1}+\sqrt{\frac{\rho_2(2+\ii\eta_1)}{N\eta_1}}\right)
	\label{eq:N-k_loc_law}
\end{equation}
with very high probability. If $\mathrm{E}_+(2)\ge 2+N^{-2/3+2\xi_0}$, then the first term in the right-hand-side of \eqref{eq:N-k_loc_law} is the leading one, so the left-hand-side of \eqref{eq:N-k_loc_law} is of order $\sqrt{\mathrm{E}_+(2)-2}\ge N^{\xi_0/2}\sqrt{\eta_1}$, which contradicts to \eqref{eq:edge_ub_aux}. In the case $\mathrm{E}_+(2)\le 2-N^{-2/3+2\xi_0}$ we get that the right-hand-side of \eqref{eq:N-k_loc_law} is stochastically dominated by $(N\eta_1)^{-1}$, which is smaller than $\sqrt{\eta_1}$ by a factor $N^{\xi_0}$. We again arrive at contradiction to \eqref{eq:edge_ub_aux}. This concludes the proof and we deduce that \eqref{eq:E_loc} must hold.
\end{proof}

\subsection{Proof of \eqref{eq:2G_special}}\label{subsec:2G_special} Our argument is based on the results proven in Appendices \ref{app:2G} and \ref{app:M}. For clarity, we restate \eqref{eq:2G_special} using the notation $D_k(E_1)$ introduced in \eqref{eq:D_E_def}
\begin{equation}
\left\langle \Im \left(H^{(N-k)}-\Re D_k(E_1) - z_1\right)^{-1} \Im (H^{(N-k)}-z_2)^{-1} \right\rangle\prec  \left\langle(\Re D_k(E_1) - \left\langle\Re D_k(E_1)\right\rangle)^2\right\rangle  ^{-1},
\label{eq:2G_special_app}
\end{equation}
where $E_1=\gamma_i^{(N)}$, $E_2=\gamma_j^{(N-k)}$, $\eta_1=\eta_2=N^{-2/3+\xi_0}$ and $z_l=E_l+\ii\eta_l$ for $l=1,2$. Exploiting the independence of the matrices $H^{(N-k)}$ and $\Re D_k(E_1)$, we condition on $D_k(E_1)$ in the left-hand-side of \eqref{eq:2G_special_app}. We interpret the resulting expression as the trace of the product of two resolvents, forming a two-resolvent chain. Note that these resolvents correspond to a Wigner matrix deformed in two different ways (actually, the second deformation equals to zero). Such quantities are analyzed in Appendix \ref{app:2G}. In particular, Proposition \ref{prop:2G} provides an upper bound on the fluctuation of a two-resolvent chain around its deterministic counterpart. This result is complemented by Proposition \ref{prop:M2_bound} which establishes an upper bound on the deterministic counterpart.

We now explain how Proposition \ref{prop:2G} is applied in the current setting. Although this result was originally stated under the assumption that $W$ is a Wigner matrix, it remains valid for $W:=H^{(N-k)}$, provided that the term $\langle M^D\rangle$ in the right-hand-side of \eqref{eq:MDE} is multiplied by $1-k/N$. This adjustment follows from the fact that $H^{(N-k)}$ multiplied by $(1-k/N)^{-1/2}$ is a Wigner matrix. Consequently, this factor can be extracted from the resolvents in the left-hand-side of \eqref{eq:2G_special_app}, allowing us to apply \eqref{eq:2G_loc}. 

The remainder of the proof is divided into two parts. First, we show that Proposition \ref{prop:2G}, together with Proposition \ref{prop:M2_bound} implies \eqref{eq:2G_special_app}. Then, we verify that the conditions of these propositions are satisfied and complete the proof of \eqref{eq:2G_special}.

We rewrite the left-hand-side of \eqref{eq:2G_special_app} as a linear combination of four terms by expanding the imaginary parts. While the two-resolvent chain with imaginary parts is typically smaller than its counterpart without imaginary parts in the regime, when the spectral parameters are near the edge (see e.g. \cite[Theorem 2.4]{cipolloni2023eigenstate}), we do not exploit this refinement here. Since the argument is identical for each term, we focus on the following representative case:
\begin{equation}
\left\langle \left(H^{(N-k)}-\Re D_k(E_1) - z_1\right)^{-1} (H^{(N-k)}-z_2)^{-1} \right\rangle.
\label{eq:GG_term}
\end{equation}
Applying Proposition \ref{prop:2G} with $W:=H^{(N-k)}$ and traceless deformations $-\Re D_k(E_1) + \langle \Re D_k(E_1)\rangle$ and $0$, we establish that the fluctuation of \eqref{eq:GG_term} around the corresponding deterministic counterpart $\langle M_{\nu_1,\nu_2}^I\rangle$ has an upper bound of order $\gamma^{-1}$. Next, Proposition \ref{prop:M2_bound} shows that $\langle M_{\nu_1,\nu_2}^I\rangle$ itself has an upper bound of order $\gamma^{-1}$. Finally, estimating $\gamma$ from below simply by the first term in the right-hand-side of \eqref{eq:gamma}, we conclude the proof of \eqref{eq:2G_special}.

To ensure that the application of Proposition \ref{prop:2G} is valid, we need to verify that for some fixed $\varepsilon, L>0$, the bound $\lVert M_{E_1}(z)\rVert\le L$ holds uniformly in $z\in\C\setminus\R$, and that $\eta_1\rho_{E_1}(z_1)\ge N^{-1+\varepsilon}$. Here we use the notation introduced in Appendix \ref{app:M}. The corresponding bounds for the deterministic counterpart of the second resolvent in \eqref{eq:GG_term} follow immediately, as it coincides with $m_{sc}$, defined above \eqref{eq:m_rigidity}, up to some explicit rescaling. For the bound $\lVert M_{E_1}(z)\rVert\le L$, we recall that $E_1=\gamma_i^{(N)}$ with $i\le N^{\varepsilon_0}$, so
\begin{equation}
0\le 2-E_1\lesssim N^{-2/3(1-\varepsilon_0)}.
\label{eq:E_1_bound}
\end{equation} 
In particular, $E_1\ge 1+\delta$ for some fixed $\delta>0$, and the bound \eqref{eq:M_bound} from Lemma \ref{lem:M_prop} implies the boundedness of $\lVert M_{E_1}(z)\rVert$. 

Lastly, we establish the bound $\eta_1\rho_{E_1}(z_1)\ge N^{-1+\varepsilon}$. Recall that $\mathrm{E}_+(E_1)$, introduced in Lemma \ref{lem:M_prop}, is the regular right edge of $\rho_{E_1}$. We compare $E_1$ with $\mathrm{E}_+(E_1)$ and consider two cases: $E_1\le \mathrm{E}_+(E_1)$ and $E_1>\mathrm{E}_+(E_1)$. In the first case $E_1\in\mathrm{supp}(\rho_{E_1})$, so \cite[Theorem 7.1]{shape} implies that
\begin{equation*}
\rho(E_1)\sim \sqrt{\mathrm{E}_+(E_1)-E_1}+\sqrt{\eta_1}\ge \sqrt{\eta_1}.
\end{equation*}
Consequently, we obtain $\eta_1\rho_{E_1}(z_1)\gtrsim \eta_1^{3/2}\ge N^{-1+3/2\xi_0}$. In the second case, where $E_1$ lies outside of the support of $\rho_{E_1}$, a similar analysis yields
\begin{equation}
\rho(E_1)\sim \frac{\eta_1}{\sqrt{E_1-\mathrm{E}_+(E_1)}+\sqrt{\eta_1}}.
\label{eq:E_1_outside}
\end{equation}
Since \eqref{eq:E_1_bound} holds, Proposition \ref{prop:edge} ensures that $\vert \mathrm{E}_+(E_1)-E_1\vert\lesssim N^{-2/3+2\xi_0}$. Substituting this bound into \eqref{eq:E_1_outside}, we conclude that $\eta_1\rho_{E_1}(z_1)\gtrsim  N^{-1+\xi_0}$, completing the proof of \eqref{eq:2G_special}.\qed

\bibliographystyle{plain} 
\bibliography{refs}

\end{document}